\documentclass[12pt]{article}
\usepackage{amsmath,amssymb,amsthm,bbm,color,xcolor}
\usepackage[margin=1in,a4paper]{geometry}
\usepackage{graphicx}
\usepackage{setspace}
\usepackage[round]{natbib}

\usepackage{tikz}
\usepackage{pgfplots}
\pgfplotsset{compat=1.18}
\usepackage{amsmath}
\usepackage{xcolor,url}

\usepackage{multirow}
\usepackage{rotating}

\newtheorem{theorem}{Theorem}
\newtheorem{lemma}[theorem]{Lemma}
\newtheorem{proposition}[theorem]{Proposition}

\usepackage[colorlinks,
  linkcolor=blue,
  citecolor=blue,
  urlcolor=blue,
  breaklinks
]{hyperref}

\newcommand{\R}{\mathbb{R}}
\newcommand{\eps}{\epsilon}
\newcommand{\EE}[1]{\mathbb{E}\left[{#1}\right]}
\newcommand{\EEst}[2]{\mathbb{E}\left[{#1}\  \middle| \ {#2}\right]}

\newcommand{\PP}[1]{\mathbb{P}\left\{{#1}\right\}}
\newcommand{\PPst}[2]{\mathbb{P}\left\{{#1}\  \middle| \ {#2}\right\}}

\newcommand{\Pp}[2]{\mathbb{P}_{{#1}}\left\{{#2}\right\}}
\newcommand{\eqd}{\stackrel{\textnormal{d}}{=}}
\newcommand{\One}[1]{{\mathbbm{1}}\left\{{#1}\right\}}

\newcommand{\iidsim}{\stackrel{\textnormal{iid}}{\sim}}

  \newcommand\independent{\protect\mathpalette{\protect\independenT}{\perp}}
\def\independenT#1#2{\mathrel{\rlap{$#1#2$}\mkern2mu{#1#2}}}

\DeclareRobustCommand{\stirling}{\genfrac\{\}{0pt}{}}

\usepackage{algorithm,algorithmic}

\newcommand{\kh}{\widehat{k}}
\newcommand{\fdp}{\mathsf{FDP}}
\newcommand{\fdr}{\mathsf{FDR}}
\newcommand{\Sbh}{\widehat{S}_{\textnormal{BH}}}
\newcommand{\nulls}{\mathcal{H}_0}

\title{False discovery rate control with compound p-values}
\author{Rina Foygel Barber\thanks{Department of Statistics, University of Chicago} \and Richard J.\ Samworth\thanks{Statistical Laboratory, University of Cambridge}}
\date{\today}

\begin{document}

\setstretch{1.7}

\maketitle

\begin{abstract}
    In the setting of multiple testing, compound p-values generalize p-values by asking for superuniformity to hold only \emph{on average} across all true nulls.  We study the properties of the Benjamini--Hochberg procedure applied to compound p-values.  Under independence, we show that the false discovery rate (FDR) is at most $1.93\alpha$, where $\alpha$ is the nominal level, and exhibit a distribution for which the FDR is $\frac{7}{6}\alpha$.  If additionally all nulls are true, then the upper bound can be improved to $\alpha + 2\alpha^2$, with a corresponding worst-case lower bound of $\alpha + \alpha^2/4$. Under positive dependence, on the other hand, we demonstrate that FDR can be inflated by a factor of $\mathcal{O}(\log m)$, where~$m$ is the number of hypotheses. We provide numerous examples of settings where compound p-values arise in practice, either because we lack sufficient information to compute non-trivial p-values, or to facilitate a more powerful analysis.
\end{abstract}
\section{Introduction}\label{sec:intro}

The Benjamini--Hochberg (BH) procedure \citep{benjamini1995controlling} is a celebrated method for achieving false discovery rate control in multiple testing problems.  Suppose we wish to test null hypotheses $H_{0,1},\dots,H_{0,m}$ (for a large number of hypotheses $m$), and aim to identify a set of discoveries---a subset of $[m]:=\{1,\dots,m\}$, corresponding to null hypotheses that we reject---without selecting too many false positives.   In order to run the BH procedure, we need access to valid p-values $p_i$ that test each of the null hypotheses $H_{0,i}$. Formally, writing $\nulls\subseteq[m]$ for the (unknown) set of indices corresponding to true null hypotheses, $p_1,\dots,p_m$ are p-values for testing hypotheses $H_{0,1},\dots,H_{0,m}$ if
\begin{equation}\label{eqn:valid_pvalues}\textnormal{$\PP{p_i\leq t}\leq t$ for all $t\in[0,1]$ and all $i\in\nulls$},\end{equation}
meaning that $p_i$ is superuniform for any true null $H_{0,i}$.

In some settings, it may be the case that $p_1,\ldots,p_m$ only satisfy a weaker notion of validity, namely that they are ``valid on average'', in the following sense:
\begin{equation}\label{eqn:compound_valid_pvalues}\textnormal{$\sum_{i\in\nulls}\PP{p_i\leq t}\leq mt$ for all $t\in[0,1]$}.\end{equation}
Random variables $p_1,\dots,p_m$ satisfying this condition are known as \emph{compound p-values}, first introduced by \citet{armstrong2022false} under a different name, and subsequently studied by \citet{ignatiadis2025empirical} and \citet{ignatiadis2024compound}.\footnote{An earlier use of the term ``compound p-value'' appears in \citet{habiger2014compound}, but this is different from the context of the above line of work: \citet{habiger2014compound}'s work proposes pooling data across different tests to construct each $p_i$, but the resulting $p_i$'s are nonetheless p-values in the sense of~\eqref{eqn:valid_pvalues}.}
Of course, by definition, we can see that p-values are a special case of compound p-values---that is, the condition~\eqref{eqn:compound_valid_pvalues} is strictly weaker than the condition~\eqref{eqn:valid_pvalues}.

The central question of this paper is to examine the theoretical properties of the BH procedure in the setting of compound p-values. 

\subsection{Background: FDR and the Benjamini--Hochberg procedure}\label{sec:background_bh}
We begin with a brief overview of some background on the Benjamini--Hochberg (BH) procedure, and its guarantees for FDR control in various settings. As in~\eqref{eqn:valid_pvalues} above, let $p_1,\dots,p_m$ be p-values for null hypotheses $H_{0,1},\dots,H_{0,m}$. Our task is to select a set of hypotheses $\widehat{S}\subseteq[m]$ for which we reject the null---these are then referred to as the ``rejections'' or ``discoveries''. The false discovery proportion (FDP) is defined as
\[\fdp(\widehat{S}) = \frac{|\widehat{S}\cap \nulls|}{1\vee |\widehat{S}|},\]
the fraction of \emph{false} discoveries (false positives) within a set of discoveries $\widehat{S}\subseteq[m]$. (The denominator is defined as $1\vee|\widehat{S}|= \max\{1,|\widehat{S}|\}$ to ensure that, if no discoveries are made, then FDP is zero.) The aim of false discovery rate (FDR) control is to ensure that 
\[\fdr:=\EE{\fdp(\widehat{S})}\leq \alpha,\]
for some prespecified target level $\alpha\in[0,1]$.

The BH procedure \citep{benjamini1995controlling} selects this set by first defining
\begin{equation}\label{eqn:BH_define_kh}
    \kh = \max\left\{k \in [m]: \sum_{i=1}^m \One{p_i\leq \frac{\alpha k}{m}} \geq k\right\},
\end{equation}
or $\kh=0$ if this set is empty, and then returning the set
\[\Sbh = \left\{i\in[m] : p_i \leq \frac{\alpha\kh}{m}\right\}.\]
By construction, $\kh = |\Sbh|$ is the number of discoveries.

There are many existing results on the BH procedure's false discovery rate, $\fdr = \EE{\fdp(\Sbh)}$.
Under independence, the BH procedure is known to provide FDR control \citep{benjamini1995controlling}:
\begin{equation}\label{eqn:bh_independent}\textnormal{For independent p-values $p_1,\dots,p_m$, \ }
\fdr = \alpha \cdot \frac{|\nulls|}{m} \leq \alpha.\end{equation}
That is, the FDR is guaranteed to be no larger than $\alpha$, the target level.
(Note that the BH procedure is actually conservative, by a factor of $\frac{|\nulls|}{m}$---the fraction of true nulls among all hypotheses being tested. A modified version of the BH procedure, proposed by \citet{storey2004strong}, estimates and then corrects for this factor.)  In fact, the same bound on FDR holds more broadly, under a positive dependence assumption on the p-values (the PRDS condition---see Section~\ref{sec:PRDS} for details) \citep{benjamini2001control}:
\begin{equation}\label{eqn:bh_prds}\textnormal{For PRDS p-values $p_1,\dots,p_m$, \ }\fdr \leq \alpha \cdot \frac{|\nulls|}{m} \leq \alpha.\end{equation}
On the other hand, if the p-values are not restricted to exhibit only positive dependence, then the FDR can be substantially worse. \citet{benjamini2001control} prove the bound:
\begin{equation}\label{eqn:bh_arbitrary}\textnormal{For arbitrarily dependent p-values $p_1,\dots,p_m$, \ }\fdr \leq \min\left\{\alpha\mathsf{h}_m \cdot \frac{|\nulls|}{m},1\right\},\end{equation}
where $\mathsf{h}_m = 1 + \dots + \frac{1}{m}\asymp \log m$ is the $m$th term in the harmonic series.  This bound cannot be improved, as verified by an explicit construction due to \citet{guo2008control}. In other words, the BH procedure may, in the worst case, increase the FDR by a factor of $\mathcal{O}(\log m)$, relative to the desired level~$\alpha$.

\subsection{Background: compound p-values and approximate compound p-values}\label{sec:compound}

Next we turn to the setting of compound p-values. Let $p_1,\dots,p_m$ be compound p-values as in~\eqref{eqn:compound_valid_pvalues}, for hypotheses $H_{0,1},\dots,H_{0,m}$. We can again run the BH procedure on these $p_i$'s---but since the $p_i$'s are no longer assumed to be p-values, what can be said about the resulting FDR control properties? 
Recent work by \citet{armstrong2022false} shows that the BH procedure satisfies
\begin{equation}\label{eqn:armstrong_FDR_bound}\textnormal{If $p_1,\dots,p_m$ are arbitrarily dependent compound p-values, \ }\fdr \leq \min\{\alpha \mathsf{h}_m, 1\}.\end{equation}
Compared to the result above for arbitrarily dependent p-values~\eqref{eqn:bh_arbitrary}, this differs only in the factor $\frac{|\nulls|}{m}$. As for the earlier result~\eqref{eqn:bh_arbitrary}, this upper bound on FDR again cannot be improved without further assumptions, due to the same type of construction as in \citet{guo2008control}.

While compound p-values are already a broader class than p-values, we can consider further relaxing the notion of validity of $p_1,\dots,p_m$. \citet{ignatiadis2024compound} define such a relaxation: $p_1,\dots,p_m$ are $(\epsilon,\delta)$-approximate compound p-values if they satisfy 
\begin{equation}\label{eqn:approx_compound_valid_pvalues}\textnormal{$\sum_{i\in\nulls}\PP{p_i\leq t}\leq m\big(t(1+\epsilon)+\delta\big)$ for all $t\in[0,1]$}.\end{equation}
For example, \citet{gasparin2025combining} refer to asymptotic p-values (where the random variables $p_i$ converge in distribution to p-values); in the finite-sample setting, the parameters~$\epsilon$ and $\delta$ relax the requirement of validity.

It is straightforward to see that, if $p_1,\dots,p_m$ are $(\eps,\delta)$-approximate compound p-values, then defining $\tilde{p}_i = \min\{1, p_i(1+\eps) +\delta \}$, it holds that $\tilde{p}_1,\dots,\tilde{p}_m$ are compound p-values as in~\eqref{eqn:compound_valid_pvalues}---that is, with a small correction factor, approximate compound p-values can be converted to compound p-values.

\subsection{Contributions and outline}

The bound~\eqref{eqn:armstrong_FDR_bound} provides an upper bound on FDR for the BH procedure when run on compound p-values, but this bound does not address settings where we may have additional assumptions on the $p_i$'s---independence, a positive dependence condition, or perhaps some other type of restriction that might allow us to avoid the $\mathcal{O}(\log m)$ inflation of the bound on FDR. While independence and, more generally, PRDS dependence lead to FDR no larger than $\alpha$ for p-values (as in~\eqref{eqn:bh_independent} and~\eqref{eqn:bh_prds} above), no such results exist in the literature for compound p-values. Is a bound such as $\fdr\leq\alpha$ possible in these settings? Or is the bound $\fdr\leq\alpha\mathsf{h}_m$, as in~\eqref{eqn:armstrong_FDR_bound}, the best we can hope for in the setting of compound p-values? 

In this work, we answer these questions, showing that in the case of independent p-values it holds that $\fdr\leq\mathcal{O}(\alpha)$, while for the PRDS setting, a bound of the form $\fdr\leq \mathcal{O}(\alpha\mathsf{h}_m)$ is the best possible bound.
Our main results addressing these questions are presented in Section~\ref{sec:main_results}. In Section~\ref{sec:examples}, we develop several examples of statistical settings where compound p-values arise, and examine how the framework of compound p-values can allow for more powerful inference by pooling information across tests. We present empirical results on a real data set of online article click rate data in Section~\ref{sec:experiments}, and conclude with a discussion of related work and open questions in Section~\ref{sec:discussion}. All proofs are deferred to the Appendix.

\section{Main results}\label{sec:main_results}
In this section, we present our main results to characterize the FDR of the Benjamini--Hochberg procedure for compound p-values.
The proofs for results stated in this section appear in Appendices~\ref{app:proofs_mainresults} and~\ref{app:proofs_constructions}.

\subsection{Results under independence}
Our first main result establishes a bound on FDR in the setting of independent compound p-values, verifying that in the worst case, the FDR can increase by at most a factor of $\leq 1.93$ relative to the target level $\alpha$.\footnote{In this section, and for all results presented in this work that assume independence of the $p_i$'s, in fact it suffices to assume a slightly weaker condition: the values $(p_i)_{i\in\nulls}$ corresponding to true nulls are mutually independent, and are independent of the non-null values $(p_i)_{i\not\in\nulls}$. This is the same as the independence assumption needed for the original FDR control result of \citet{benjamini1995controlling}.}
\begin{theorem}[Upper bound for independent compound p-values]\label{thm:fdr_independent}
    Fix any $\alpha\in[0,1]$, any $m \geq 1$, and any $\nulls\subseteq[m]$. Let $p_1,\dots,p_m\in[0,1]$ be independent compound p-values. Then the BH procedure at level $\alpha$ satisfies
    \[\fdr \leq 1.93\alpha.\]
\end{theorem}

In comparison, the only previously known finite-sample bound on FDR for compound p-values is shown in~\eqref{eqn:armstrong_FDR_bound}, which bounds FDR by $\alpha\mathsf{h}_m$, since it does not make use of the independence assumption. We can see that the independence assumption allows us to avoid the possibility of a $\mathcal{O}(\log m)$ inflation of the FDR.

Of course, we might ask whether the bound $1.93\alpha$ can be improved. 
We now see that the fact that the FDR bound is higher than the target level $\alpha$ is sometimes unavoidable: in fact, the increase in FDR can be proportional to $\alpha$.
\begin{proposition}[Lower bound for independent compound p-values]\label{prop:fdr_example}
    Fix any $\alpha\in[0,1]$ such that $1/(2\alpha)$ is an integer, and fix any integer $m\geq 3/(2\alpha)$. Then there exists a subset $\nulls\subseteq[m]$ and a distribution for independent compound p-values $p_1,\ldots,p_m$, such that the BH procedure at level~$\alpha$ satisfies
    \[\fdr  \geq \tfrac{7}{6}\alpha.\]
\end{proposition}
It was previously known that a bound of the form $\fdr\leq \alpha$ is impossible for independent compound p-values: indeed, \citet[Section 3.1]{armstrong2022false} gives a counterexample with independent compound p-values for which $\fdr \geq \alpha + \frac{\alpha^2}{4m^2}$, which is a positive (although vanishing) inflation of the FDR above the target level $\alpha$. Our new result strengthens this finding, by verifying that a constant factor inflation is unavoidable, 
in the worst case.\footnote{See Appendix~\ref{app:extensions} for a result showing that a tighter upper bound on FDR becomes possible with an additional assumption on the marginal distributions of the $p_i$'s.}

Theorem~\ref{thm:fdr_independent} can also be extended to the setting of approximate compound p-values.
\begin{theorem}[Upper bound for independent approximate compound p-values]\label{thm:fdr_approx}
    Fix any $\alpha\in[0,1]$, any $m \geq 1$, and any $\nulls\subseteq[m]$. Let $p_1,\dots,p_m\in[0,1]$ be independent $(\eps,\delta)$-approximate compound p-values, where $\eps,\delta\geq 0$. If $\delta=0$, then the BH procedure at level $\alpha$ satisfies
    \[\fdr\leq 1.93\alpha(1+\eps).\]
    If instead $\delta>0$, then a similar bound holds for a modified version of the FDR:
    \[\EE{\frac{|\Sbh\cap\nulls|}{\frac{m\delta}{\alpha(1+\eps)} + |\Sbh|}} \leq 1.93\alpha(1+\eps).\]
\end{theorem}
In this last bound, we note that the quantity in the expected value is a modified version of the FDP: recalling that $\fdp(\Sbh)=\frac{|\Sbh\cap\nulls|}{1\vee|\Sbh|}$, we see that (for small $\delta$) the modification is small as long as the number of discoveries, $|\Sbh|$, is large.

\subsection{Results under independence, and under the global null}
For independent compound p-values, the results above show that the FDR of the BH procedure can be bounded by a constant-factor inflation (Theorem~\ref{thm:fdr_independent}), and that a constant-factor inflation is unavoidable (Proposition~\ref{prop:fdr_example}).

Interestingly, though, we will now see that the situation improves further under the global null, where $\nulls=[m]$---that is, all $m$ null hypotheses are true. 
In this setting, the false discovery rate is equal to the probability of making any discoveries, since $\fdp(\Sbh)=\frac{|\Sbh\cap\nulls|}{1\vee|\Sbh|} = \One{\Sbh\neq\varnothing}$ (that is, the FDR is equal to the family-wise error rate,
$\mathsf{FWER} = \PP{|\Sbh\cap\nulls|>0}$, since any discovery would be a false discovery). For this setting, we will now see that the inflation in the FDR can be controlled with a tighter bound: in the worst-case, the excess FDR (i.e., the amount by which the FDR can exceed the target level $\alpha$) is $\mathcal{O}(\alpha^2)$, rather than $\mathcal{O}(\alpha)$ as in Theorem~\ref{thm:fdr_independent} above.

\begin{theorem}[Upper bound for independent compound p-values, under global null]\label{thm:fdr_globalnull}
    Fix any $\alpha\in[0,1]$ and any $m\geq 1$, and let $\nulls=[m]$. Let $p_1,\dots,p_m\in[0,1]$ be independent compound p-values. Then the BH procedure at level $\alpha$ satisfies
    \[\mathsf{FWER} = \fdr \leq \alpha + 2\alpha^2.\]
\end{theorem}

Next, we verify that this $\mathcal{O}(\alpha^2)$ inflation of the bound on FDR is unavoidable---up to a constant, this is optimal for the setting of Theorem~\ref{thm:fdr_globalnull}.

\begin{proposition}[Lower bound for independent compound p-values, under global null]\label{prop:fdr_example_globalnull}
    Fix any $\alpha\in[0,\frac{2}{3}]$, any integer $m\geq 2$, and let $\nulls = [m]$. Then there exists a distribution for independent compound p-values $p_1,\dots,p_m$, such that the BH procedure at level $\alpha$ satisfies
    \[\mathsf{FWER} = \fdr \geq \alpha + \tfrac{1}{4}\alpha^2.\]
\end{proposition}

\subsection{Results under positive dependence (PRDS)}\label{sec:PRDS}
Our main results thus far have all assumed independence of the values $p_1,\dots,p_m$. We now turn our attention to a more general setting, where $p = (p_1,\dots,p_m)$ are assumed to satisfy the condition of \emph{positive regression dependence on a subset} (PRDS), introduced by \citet{benjamini2001control}:\footnote{In fact, the PRDS condition was originally defined in a stronger sense: the function $t\mapsto \PPst{p\in A}{p_i= t}$ is nondecreasing in $t$ (that is, we condition on $p_i=t$ rather than $p_i\leq t$), but the version given in~\eqref{eqn:PRDS} suffices for verifying FDR control \citep[Section~4]{benjamini2001control}.}
\begin{multline}\label{eqn:PRDS}
\textnormal{For any increasing set $A\subseteq[0,1]^m$, and every $i\in\nulls$,}\\\textnormal{the function $t\mapsto \PPst{p\in A}{p_i\leq t}$ is nondecreasing in $t$.}
\end{multline}
(A set $A\subseteq[0,1]^m$ is said to be \emph{increasing} if, whenever $x \in A$ and $y \in [0,1]^m$ satisfy $x \leq y$ coordinatewise, we have $y\in A$.) Informally, this condition requires that a higher value of $p_i$ should not be correlated with lower values of the other $p_j$'s.

As mentioned in Section~\ref{sec:background_bh} (see~\eqref{eqn:bh_prds}), if $p_1,\dots,p_m$ are p-values and satisfy the PRDS condition, then the BH procedure is known to control FDR at the level $\alpha$ \citep{benjamini2001control}. 
In contrast, if instead $p_1,\dots,p_m$ are compound p-values, then the only known upper bound is given by~\eqref{eqn:armstrong_FDR_bound}, which does not make use of the restriction on the dependence. We might hope that, similarly to the setting of p-values, the PRDS condition can allow for substantially better control of the FDR, and can avoid the factor of $\mathcal{O}(\log m)$---but in fact,
our results verify that this factor cannot be avoided.
\begin{proposition}[Lower bound for PRDS compound p-values]\label{prop:fdr_example_prds}
    Fix any $\alpha\in[0,1]$, any integer $m\geq 1$ and let $\nulls=[m]$. Then there exists a distribution on $p=(p_1,\dots,p_m)\in[0,1]^m$, where $p_1,\dots,p_m$ satisfy the PRDS condition~\eqref{eqn:PRDS} and are compound p-values, such that the BH procedure at level $\alpha$ satisfies
    \[\fdr\geq \tfrac{3}{8} \min\{\alpha \mathsf{h}_m,1\}.\]
\end{proposition}
\noindent In particular, this shows that \citet{armstrong2022false}'s upper bound~\eqref{eqn:armstrong_FDR_bound} on FDR for compound p-values is, up to a constant, optimal even under the PRDS condition.

\section{Compound p-values: examples}\label{sec:examples}
In this section, in order to motivate the problem of multiple testing with compound p-values, we present several examples of settings where compound p-values might arise.  We will consider two types of examples:
\begin{itemize}
    \item Settings where there is insufficient information to compute non-trivial p-values satisfying the usual validity condition, but compound p-values are nonetheless possible to construct.
    \item Settings where we do have sufficient information to construct p-values, but we instead choose to use compound p-values to enable a more powerful analysis---for example, because the only available p-values are too coarse, too noisy, or too conservative to be informative.
\end{itemize}
In the specific settings described below, Examples~\ref{example:empirical_Bayes} and~\ref{example:decreasing_densities} are of the first type, while Examples~\ref{example:permutation_tests} and~\ref{example:MC_pvalues} are of the second type. Additional examples are provided in Appendix~\ref{app:additional_examples}.
Proofs of all results stated in this section are presented in Appendix~\ref{app:proofs_examples}.

\newcounter{examplesettings}
\renewcommand\theexamplesettings{(\alph{examplesettings})}

\refstepcounter{examplesettings}\label{example:empirical_Bayes}
\subsection{Example \theexamplesettings: empirical Bayes and the two groups model} 

For a first example, we will consider a classical empirical Bayes approach to hypothesis testing, namely the two-groups model; see \citet{efron2012large} for background. In this approach, we assume that each data value $X_i$ is drawn from a mixture of two densities, $f=\pi_0\cdot f_0 + (1-\pi_0)\cdot f_1$, where $f_0$ is the density of a data point drawn under the null, and $f_1$ is the density under the alternative. 

For simplicity, suppose large (and positive) values of $X_i$ indicate evidence against the null, and consider a rejection threshold $z$---that is, consider rejecting hypotheses $\{i: X_i \geq z\}$. We can then estimate the the false discovery proportion of this rejection rule as $\widehat{\fdp}(z)= \frac{\pi_0\cdot \Pp{f_0}{X\geq z}}{\frac{1}{m}\sum_{i=1}^m \One{X_i\geq z}}$ (where the denominator can be viewed as an estimate of $\bar{F}(z) = \Pp{f}{X\geq z}$, since $X_1,\dots,X_m$ are draws from the mixture density $f$).\footnote{In fact, it is well known that this approach to estimating FDP can be viewed as an interpretation of the BH procedure \citep{efron2002empirical,storey2004strong}---if we define $p_i = \bar{F}_0(X_i)$ and set $\pi_0=1$ as a conservative upper bound on the proportion of true nulls, then the set of rejections returned by the BH procedure is $\Sbh = \{i\in[m]: X_i \geq \hat{z}\}$, where $\hat{z} = \min\{ z: \widehat{\fdp}(z) \leq \alpha\}$.}

The term in the numerator, $\bar{F}_0(z)=\Pp{f_0}{X\geq z}$, is the right-tail CDF of the null density $f_0$, and may be interpreted as a p-value---more formally, if we have a null data value $X_i\sim f_0$, then $p_i = \bar{F}_0(X_i)$ is a p-value. But
\citet{efron2007size} observes that, for the estimate $\widehat{\fdp}(z)$ to be meaningful, we do not need to require ``all of the null [$X_i$'s]
to have the same density $f_0(z)$, only that their \emph{average} density behaves like $f_0$''. In other words,~$\bar{F}_0$ only needs to be a reasonable estimate of the \emph{average} of the right-tail CDFs of the null distributions. With this observation, we can now connect the empirical Bayes approach directly to compound p-values.

\begin{proposition}\label{prop:average_null_CDFs}
    Suppose we observe data $X_i\in\R$ to test hypothesis $H_{0,i}$, for each $i\in[m]$.
    Let $F_i(x) = \Pp{H_{0,i}}{X_i\leq x}$ be the CDF of the null distribution of $X_i$. If $p_1,\dots,p_m$ satisfy
    \[p_i \geq \frac{1}{m}\sum_{j\in\nulls} F_j(X_i)\textnormal{ almost surely for each $i\in\nulls$},\]
    then $p_1,\dots,p_m$ are compound p-values.
    Similarly, if $\bar{F}_i(x)=\Pp{H_{0,i}}{X_i\geq x}$ denotes the right-tailed CDF of the null distribution of $X_i$, and if $p_1,\dots,p_m$ satisfy
    \[p_i \geq \frac{1}{m}\sum_{j\in\nulls} \bar{F}_j(X_i)\textnormal{ almost surely for each $i\in\nulls$},\]
    then again $p_1,\dots,p_m$ are compound p-values.
\end{proposition}

\refstepcounter{examplesettings}\label{example:decreasing_densities}
\subsection{Example \theexamplesettings: decreasing densities}
Suppose we observe data values $X_1,\dots,X_m$, where each $X_i$ is independently drawn from a distribution $P_i$ on $[0,\infty)$. We will take $H_{0,i}$ to be the null hypothesis that $P_i$ has a nonincreasing density on $[0,\infty)$---under the alternative, we might instead have a density with its mode at a positive value.

Clearly, for a single data point $X_i$, we have no ability to perform inference: if we observe, say, $X_i = 1.5$, we cannot state that this is inconsistent with the null hypothesis of a nonincreasing density, without further information.  On the other hand, by sharing information across the $m$ tests, we can define compound p-values. There are many possible ways to do this---here we consider one simple possibility, but this may be far from optimal. 

Let $\mathcal{G}$ be the set of convex and nonincreasing functions $g:[0,\infty) \to [0,1]$, and for any $\Delta>0$, define
\[\mathcal{G}(\Delta)=\left\{ g\in\mathcal{G} :  \max_{i=0,\dots,m} \{g(X_{(i)}) - g(X_{(i+1)})\}\leq \Delta\right\},\]
where $X_{(1)}\leq \dots \leq X_{(m)}$ are the order statistics of the observed data values, and we also define $X_{(0)}=0$ and $X_{(m+1)} = +\infty$, and let $g(+\infty)=0$. Finally, let
\begin{equation}\label{eqn:define_p_i_decreasing_densities}p_i = \sup_{g\in\mathcal{G}(\Delta)} g(X_i)\textnormal{ where }\Delta = \frac{\log(em^2)}{m}.\end{equation}
(See Appendix~\ref{app:decreasing_density_appendix} for an illustration.)

The intuition for this construction is the following. Let $g^*(x) =\frac{1}{m}\sum_{j\in\nulls}\bar{F}_j(x)$, where for each $j\in\nulls$, $\bar{F}_j$ denotes the right-tail CDF of the distribution $P_j$. By Proposition~\ref{prop:average_null_CDFs}, then, the quantities
    \[p^*_i = g^*(X_i)\]
define compound p-values---but we cannot compute $p^*_i$ since $g^*$ is unknown (because the null distributions $P_j$ are themselves unknown). However, with high probability, the unknown function $g^*$ lies in $\mathcal{G}(\Delta)$---and on this event, we therefore have $p_i\geq p^*_i$, i.e., the values $p_i$ are more conservative than the compound p-values $p^*_i$.

The following result validates this intuition, and verifies that our proposed method yields approximately compound p-values.
\begin{proposition}\label{prop:decreasing_density}
    Let $X_1\sim P_1$, \dots, $X_m\sim P_m$ be independent, where each $P_i$ is a distribution on $[0,\infty)$. Let $H_{0,i}$ be the null hypothesis that $P_i$ has a nonincreasing density.
    Then $p_1,\dots,p_m$ (as defined above) are $(0,\frac{1}{m})$-approximate compound p-values.
\end{proposition}

Finally, while the definition of the $p_i$'s may initially appear challenging to compute due to the function class $\mathcal{G}(\Delta)$, we now show that in fact the supremum can be derived with a simple calculation:
\begin{proposition}\label{prop:decreasing_density_compute}
    Under the notation and assumptions of Proposition~\ref{prop:decreasing_density}, define $w_m=\Delta$ and recursively for each $i=m-1,\dots,0$, set
    \[w_i = \min\left\{\frac{1-\sum_{j=i+1}^m w_j}{ \frac{X_{(i+1)}}{X_{(i+1)} - X_{(i)}}} 
    , \frac{\Delta}{\max_{j=0,\dots,i} \frac{X_{(j+1)} - X_{(j)}}{X_{(i+1)} - X_{(i)}}}\right\},\]
    or $w_i=0$ if $X_{(i+1)}-X_{(i)}=0$.
    Then for each $i\in[m]$, 
    \[p_i = \sup_{g\in\mathcal{G}(\Delta)}g(X_{(i)})  = \sum_{j=i}^m w_j .\]
\end{proposition}

\refstepcounter{examplesettings}\label{example:permutation_tests}
\subsection{Example \theexamplesettings: permutation tests} 
\label{subsec:permutationtests}
Suppose that we are analyzing $m$ many A/B trials---for instance, clinical trials for $m$ medications. In the $i$th trial, there are $n_i$ units (i.e., patients in the clinical trial), with $n_{i1}$ assigned to the treatment group and $n_{i0} = n_i-n_{i1}$ to the control group. Let $X_i=(X_{i1},\dots,X_{in_i})$ denote the observed data for the $i$th trial (where, without loss of generality, we assume that the data $X_i$ is arranged so that the first $n_{i1}$ entries are the treated units, and the remaining~$n_{i0}$ entries are the control units). Our test statistic for the $i$th trial is given by $T_i(X_i)$, for some function $T_i$---for example, a two-sample t-statistic comparing the data $X_{i1},\dots,X_{in_{i1}}$ for the treated group and $X_{i(n_{i1}+1)},\dots,X_{in_i}$ for the control group, in the case of real-valued data. 

If the sample sizes $n_i$ are relatively small then we often would not want to rely on the asymptotic validity of a t-test, and instead might use a permutation test to produce p-values
\begin{equation}
\label{eq:permutationpvalues}
p^{\textnormal{perm}}_i = \frac{1}{n_i!}\sum_{\sigma\in\mathcal{S}_{n_i}} \One{T_i(X_i^\sigma) \geq T_i(X_i)},
\end{equation}
where $\mathcal{S}_{n_i}$ is the set of permutations on $[n_i]$, and $X_i^\sigma$ denotes the vector $X_i$ permuted by~$\sigma$.
However, when these sample sizes $n_i$ are fairly small for each trial $i$, these p-values are very coarse: for example, if the $i$th trial has $n_i=8$ units, then $p^{\textnormal{perm}}_i$ cannot be smaller than $\frac{1}{\binom{8}{4}}\approx 0.014$. This means that a multiple testing procedure will often be unable to make any discoveries, if the number of trials $m$ is large.\footnote{For discrete problems where p-values take values on a discrete grid $\{\frac{1}{K},\frac{2}{K},\dots,1\}$, it is possible to use randomization (or ``smoothing'') to provide continuous p-values---that is, if $p^{\textnormal{perm}}_i = \frac{k}{K}$, we would then return a draw from $\textnormal{Unif}((\frac{k-1}{K},\frac{k}{K}])$. Nonetheless, the same issue of low power holds: regardless of signal strength, it is impossible to construct a p-value that is extremely low with high probability.} 

To avoid this issue of low power, we can instead turn to compound p-values. Specifically, we define
\begin{equation}
\label{eq:compoundpvaluesfrompermutation}
p_i = \frac{1}{m}\sum_{j=1}^m \Biggl\{ \frac{1}{n_j!}\sum_{\sigma\in\mathcal{S}_{n_j}} \One{T_j(X_j^\sigma) \geq T_i(X_i)}\Biggr\}.
\end{equation}
This type of approach for pooling permutation tests is proposed in \citet[Section 18.7]{hastie2009elements}---the strategy is, in flavor, an empirical Bayes-style approach: the test statistic $T_i(X_i)$ for the $i$th trial is compared against permuted test statistics $T_j(X_j^\sigma)$ for \emph{every} trial $j\in[m]$, rather than only comparing against permutations of the same trial $i$ (\citet{hastie2009elements} describe this as comparing $T_i(X_i)$ against a ``pooled null distribution''). 
For example, if $T_i(X_i)$ is a two-sample t-statistic, then we would expect each $T_i(X_i)$ to have approximately the same null distribution (asymptotically, a t distribution), and so this comparison is well motivated. However, we emphasize that the following result holds \emph{without} assuming any distributional properties---even if the sample sizes $n_i$ for each trial are small (i.e., if the t distribution is a poor approximation).

\begin{proposition}\label{prop:perm_test_compound}
    Assume that 
    for each $i\in\nulls$, we have $X_{i1},\dots,X_{in_i}\iidsim P_i$ for some distribution $P_i$, and assume also that the data vectors $X_1,\dots,X_m$ are mutually independent. Then $p_1,\dots,p_m$ are compound p-values.
    Moreover, defining $\widehat{P}_i = \frac{1}{n_i}\sum_{\ell=1}^{n_i}\delta_{X_{i\ell}}$ as the empirical distribution of $X_i$ for each $i\in[m]$, conditionally on $\widehat{P}_1,\dots,\widehat{P}_m$ it holds that $p_1,\dots,p_m$ are independent compound p-values.
\end{proposition}

To understand the value of the second claim, recall that, from our main results in Section~\ref{sec:main_results}, independence among the $p_i$'s is important in order to establish FDR control for the BH procedure (without a $\mathcal{O}(\log m)$ inflation factor). In this example, it is clear that the $p_i$'s are \emph{not} independent, since each $p_i$ depends on every data vector $X_j=(X_{j1},\dots,X_{jn_j})$ (over all $j\in[m]$). However, we can still establish independence, after conditioning on the order statistics---and therefore, the BH procedure will be guaranteed to satisfy the bound on FDR established in Theorem~\ref{thm:fdr_independent}.

\refstepcounter{examplesettings}\label{example:MC_pvalues}
\subsection{Example \theexamplesettings: Monte Carlo p-values} 
We next develop another example where it is possible to compute p-values, but where these p-values are very coarse and lead to low power. Consider a setting where we observe data~$X_i$ for the $i$th hypothesis $H_{0,i}$, and the null distribution is computationally intractable, so that we are only able to compute a p-value via Monte Carlo sampling. 
Specifically, letting $X_{i1},\dots,X_{iK}$ be i.i.d.\ draws from the null distribution specified by $H_{0,i}$, we can define a Monte Carlo p-value
\[p^{\textnormal{MC}}_i = \frac{1 + \sum_{k=1}^K \One{T_i(X_{ik};\widehat{P}_i)\geq T_i(X_i;\widehat{P}_i)}}{1+K},\]
where $T_i$ is some function defining the test statistic, and where $\widehat{P}_i = \frac{1}{1+K}\bigl(\delta_{X_i}+\sum_{k=1}^K\delta_{X_{ik}}\bigr)$ is the empirical distribution of $(X_i,X_{i1},\dots,X_{iK})$---the real data and the Monte Carlo samples for the $i$th test. For example, we might define $T_i(X_i;\widehat{P}_i)$ as a t-statistic, 
\[T_i(X_i;\widehat{P}_i) = \left|\frac{X_i - \hat\mu_i}{\hat\sigma_i}\right|\]computed by comparing $X_i$ against the sample mean $\hat\mu_i$ and sample variance $\hat\sigma^2_i$ of the values $(X_i,X_{i1},\dots,X_{iK})$.

The values $p^{\textnormal{MC}}_i$ are valid p-values---but if $m$ is large, then we might only be able to afford a relatively small number $K$ of Monte Carlo draws for testing each null. Consequently, since these p-values cannot be smaller than $\frac{1}{K+1}$, we will be unable to achieve good power in a multiple testing procedure.

Instead, we turn to compound p-values to resolve this issue. Define
\[p_i = \frac{1}{m}\sum_{j=1}^m \left\{\frac{\One{T_j(X_j;\widehat{P}_j)\geq T_i(X_i;\widehat{P}_i)} + \sum_{k=1}^K \One{T_j(X_{jk};\widehat{P}_j)\geq T_i(X_i;\widehat{P}_i)}}{1+K}\right\}.\]
\begin{proposition}\label{prop:MC_compound}
Assume that for each $i\in\nulls$, we have $X_i,X_{i1},\dots,X_{iK}\iidsim P_i$ for some distribution $P_i$. Assume also that the data vectors $(X_i,X_{i1},\dots,X_{iK})$ are mutually independent across $i\in[m]$. 
Then $p_1,\dots,p_m$ are compound p-values.
Moreover, conditionally on $\widehat{P}_1,\dots,\widehat{P}_m$ it holds that $p_1,\dots,p_m$ are independent compound p-values.
\end{proposition}

\section{Online article headline data}\label{sec:experiments}

We present an example of a data set that falls into the category of Example~\ref{example:permutation_tests} (Section~\ref{subsec:permutationtests}).  The \texttt{Upworthy} data set \citep{matias2021upworthy} records click rates for each of $22{,}739$ online articles.\footnote{The data are available at \url{https://upworthy.natematias.com/}. Code to reproduce the experiment is available at \url{https://github.com/rinafb/FDR_compound/}.}  Different versions of headlines are often employed, to assess effectiveness at increasing user engagement.  In total, there are $64{,}051$ headlines, with the number of headlines tested for each article ranging between $1$ and $22$, and having a median of $3$.  For each variant, data are recorded on the number of ``impressions'' (times that a user is presented with the headline), and the number of ``clicks'' (i.e., for how many of the impressions did the user choose to click on the article). 

Online article headlines commonly start with, or contain, numerical digits (e.g., headlines such as ``$5$ simple reasons that...'' or ``How to ... in $1$ minute''). In this experiment, we are interested in whether or not the presence of any numerical digits in the headline increases the click rate; this type of effect has been noted in many studies on various data sets, e.g., \citet{biyani20168}. 
We restrict attention to the $m=3{,}893$ articles, with $17{,}033$ headlines in total, for which there are versions of the headlines that do and do not have a numerical digit; among this subset of the data, the median number of headlines per article is $6$.

\begin{figure}[t]
    \includegraphics[width=\textwidth]{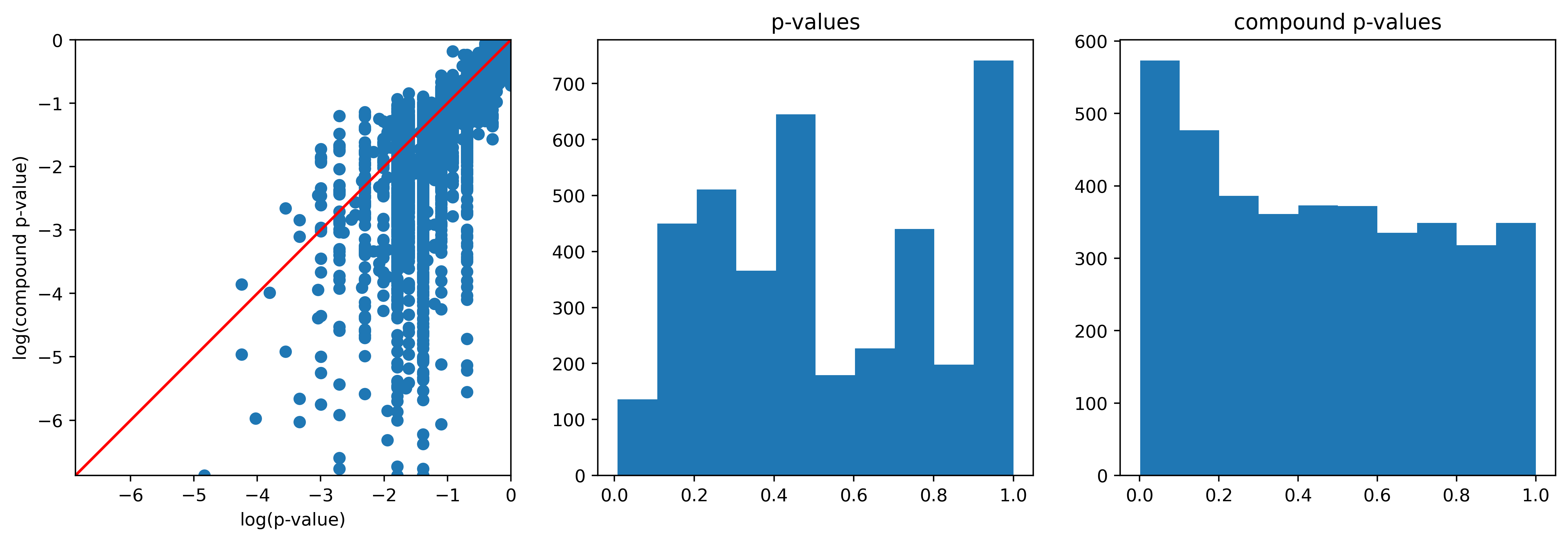}
    \caption{Visualization of the p-values, and compound p-values, obtained for the Upworthy data set. See Section~\ref{sec:experiments} for details.}
    \label{fig:upworthy_0}
\end{figure}

\begin{figure}[t]
    \includegraphics[width=\textwidth]{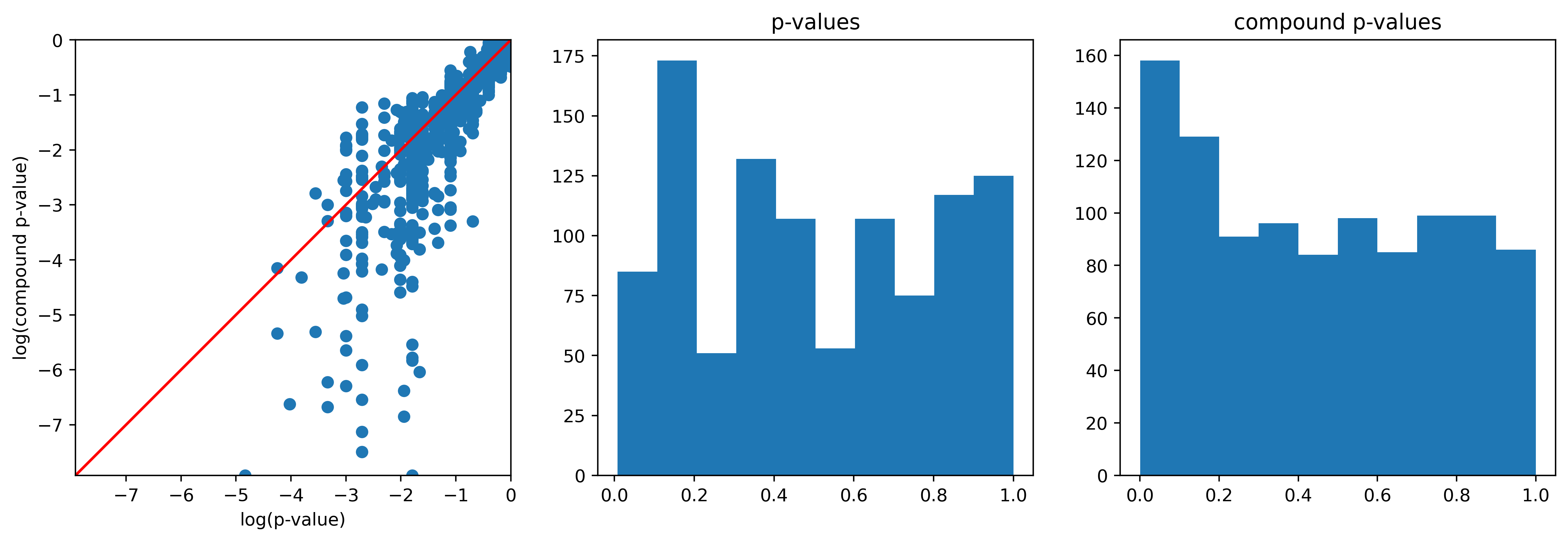}
    \caption{Visualization of the p-values, and compound p-values, obtained for the Upworthy data set, when only articles with more than five headlines are included. See Section~\ref{sec:experiments} for details.}
    \label{fig:upworthy_1}
\end{figure}

\paragraph{The data.}
The data consist of values
\[X_{ij} = (X_{ij}^{\textnormal{impr}},X_{ij}^{\textnormal{click}}),\]
for article $i=1,\dots,m$ and headline $j=1,\dots,n_i$. Headlines $j=1,\dots,n_{i1}$ contain a numerical digit, while headlines $j=n_{i1}+1,\dots,n_i$ do not. Here $X_{ij}^{\textnormal{impr}}$ is the number of impressions for the $j$th version of the $i$th article's headline, and $X_{ij}^{\textnormal{click}}$ is the corresponding number of clicks. As in Example~\ref{example:permutation_tests} above, the null $H_{0,i}$ is the hypothesis that data points $X_{i1},\dots,X_{in_i}$ are i.i.d.\ (or more generally, exchangeable) for article $i$---that is, the presence or absence of a numerical digit in a headline is not associated with any systematic change in the click rate for the $i$th article. In particular, it may be the case that headlines $j=1,\dots,n_i$ have varying click rates (as might occur in a random effects model, say), but these click rates do not depend on the presence of a numerical digit.

\paragraph{Defining a test statistic.}
Our test statistic is defined by pooling together all articles with, and without, numerical digits, and performing the (one-sided) Fisher's exact test on the resulting $2\times 2$ contingency table:
\[T_i(X_i) = 1-p_{\textnormal{Fisher}}\left( \left[\begin{array}{cc} \sum_{j=1}^{n_{i1}} X_{ij}^{\textnormal{click}}& \sum_{j=1}^{n_{i1}}  X_{ij}^{\textnormal{no-click}}\\ 
\sum_{j=n_{i1}+1}^{n_i}X_{ij}^{\textnormal{click}}& \sum_{j=n_{i1}+1}^{n_i} X_{ij}^{\textnormal{no-click}}\end{array}\right]\right),\]
where $p_{\textnormal{Fisher}}(\cdot)$ denotes the one-sided p-value for Fisher's exact test, and where $X_{ij}^{\textnormal{no-click}} = X_{ij}^{\textnormal{impr}} - X_{ij}^{\textnormal{click}}$ is the number of impressions that did not lead to a click for the $j$th headline variant tested for the $i$th article. (We define the test statistic as $1-p_{\textnormal{Fisher}}(\cdot)$, rather than $p_{\textnormal{Fisher}}(\cdot)$, to follow the convention that large values of $T_i(X_i)$ indicate evidence against the null.)\footnote{We remark that we would not expect Fisher's exact test to provide a valid p-value. This is because, for any given article $i$, each cell of the $2\times 2$ contingency table is pooling data from multiple articles---for instance, if $n_i=6$, we might have $3$ articles with a numerical digit, and $3$ without. Even if the presence of a numerical digit has no systematic effect, heterogeneity among these various headlines might mean that the click rate differs across the $6$ headlines, and the observed click rates therefore exhibit a more complex structure than uniformly random sampling.}

\paragraph{Results.} Using the test statistic defined above, permutation p-values $p^{\textnormal{perm}}_i$ were computed as in~\eqref{eq:permutationpvalues}, and compared with the compound p-values $p_i$ from~\eqref{eq:compoundpvaluesfrompermutation}. The results are shown in Figure~\ref{fig:upworthy_0}. We can see that the p-values defined as in~\eqref{eq:permutationpvalues} lack power on this dataset. This is due to the small number of different headlines for each article---since the median value of $n_i$ is~$6$, a permutation test on this sample size cannot lead to a low p-value. In contrast, the compound p-value framework is able to borrow strength across the whole suite of tests to increase power, with many low values of $p_i$. In particular, the first panel of Figure~\ref{fig:upworthy_0} shows that the compound p-value $p_i$ is often substantially smaller than its p-value counterpart $p^{\textnormal{perm}}_i$.

Since low values of $n_i$ naturally lead to low power as described above, we also repeat this experiment with a smaller subset of the data, keeping only those articles with $>5$ headlines (i.e., $n_i>5$). After this filtering step, $m=1{,}025$ articles remain, with $8{,}705$ headlines in total, and a median of $8$ headlines per article. We then calculate the p-values and compound p-values for this reduced data set, exactly as before, and display these results in Figure~\ref{fig:upworthy_1}.

Finally, we apply the BH procedure to the p-values and compound p-values obtained from each experiment, with results shown in Table~\ref{tab:upworthy}, which confirms that, in this example, the increased power of the compound p-values can enable greater discoveries. (We remark that, according to the guarantee of Theorem~\ref{thm:fdr_independent}, when using compound p-values the FDR for the BH procedure at $\alpha=0.2$ might potentially be as high as $1.93\cdot 0.2$. Nonetheless, on this dataset, using compound p-values is still more powerful than using p-values, since the BH procedure is unable to make any discoveries with p-values even at $\alpha =0.5$.)

\begin{table}[t]
\centering
\begin{tabular}{c||c|c||c|c}
& \multicolumn{2}{c||}{All articles}& \multicolumn{2}{c}{Articles with $>5$ headlines}\\\hline
& \multirow{2}{*}{p-values} & compound& \multirow{2}{*}{p-values} & compound\\
& & p-values & & p-values\\\hline
BH with $\alpha = 0.1$ &0&0&0&0\\\hline
BH with $\alpha = 0.2$ &0&0&0&11\\\hline
BH with $\alpha = 0.5$ &0&57&0&36
\end{tabular}
\caption{Number of discoveries obtained from the Benjamini--Hochberg procedure, applied to p-values or compound p-values computed on the Upworthy data set. See Section~\ref{sec:experiments} for details. (Note that BH with p-values guarantees $\fdr\leq \alpha$, while BH with compound p-values offers a weaker guarantee $\fdr\leq 1.93\alpha$ in the worst case.)}
\label{tab:upworthy}
\end{table}

\section{Discussion}\label{sec:discussion}

In this work, we have established bounds on the false discovery rate of the BH procedure in the setting of compound (or approximately compound) p-values. Our results show that, under independence, the FDR can be bounded as $\mathcal{O}(\alpha)$, but a bound of $\alpha$ is not possible (without further assumptions). In contrast, under the PRDS condition, the previously known bound on FDR given by $\alpha\mathsf{h}_m$ \citep{armstrong2022false} cannot be improved beyond a constant factor, which is qualitatively very different from the results for p-values where the PRDS condition is sufficient to ensure the same FDR bound as under independence \citep{benjamini1995controlling}. We have also developed new examples of problems where compound p-values enable powerful inference, including some settings in which p-values would be (nearly) powerless.

We conclude with a brief overview of some related works in the literature, and a discussion of some possible extensions and open questions suggested by this work.

\subsection{Additional related work}

As discussed in Section~\ref{sec:intro}, the idea of compound p-values was introduced by \citet{armstrong2022false}, under the name of ``average significance level controlling tests'', and further studied by \citet{ignatiadis2025empirical} and \citet{ignatiadis2024compound} under the name ``compound p-values''. The notion of compound p-values is closely related to the idea of confidence intervals for multiple parameters of interest (say, a confidence interval $C_j$ for a parameter $\theta_j$, for each $j\in[m]$), which are only required to have coverage in an average sense, i.e., $\frac{1}{m}\sum_{j=1}^m\PP{\theta_j\in C_j}\geq 1-\alpha$. This type of property for confidence intervals is studied by \citet{armstrong2022robust}, where the construction is motivated by an empirical Bayes argument but enjoys finite-sample (average) coverage properties. A related line of earlier work is on confidence bands for an unknown target function $f(t)$, where coverage is averaged over $t$ rather than required to hold pointwise; see \citet{cai2014adaptive} and \citet[Chapter 5.8]{wasserman2006all}.

The FDR control results studied in this work are all finite-sample results, which hold regardless of the values of $m$ and of the number of non-nulls $m-|\nulls|$, but other works examine FDR control from an asymptotic viewpoint. For the BH procedure on p-values, the work of \citet{storey2002direct} and \citet{storey2004strong} shows that the FDP converges to $\alpha$, under asymptotic assumptions on the empirical distribution of the p-values (in particular, the empirical distribution of the nulls should converge to the uniform distribution). \citet{armstrong2022false} establishes analogous results for compound p-values.

Finally, an alternative to the framework of p-values is the notion of \emph{e-values}, which have been increasingly studied in the recent statistics literature; see \citet{ramdas2024hypothesis} for an overview of the background and current research on this topic. An e-value $e_i$ for testing a null hypothesis $H_{0,i}$ is a nonnegative random variable satisfying $\EE{e_i}\leq 1$ if the null is true (in contrast, a p-value satisfies $\PP{p_i\leq t}\leq t$ if the null is true). \citet{wang2022false} develop the e-BH procedure, which is an analogue of the BH procedure for e-values rather than p-values, and establish that the FDR is controlled at level $\alpha$ even under the relaxed assumption of compound e-values rather than (individually valid) e-values. This stands in contrast to the setting of p-values, where we have seen that FDR control guarantees for p-values can fail to hold for compound p-values. Compound e-values are studied further by \citet{ignatiadis2024compound}, who also provide mechanisms for converting between compound p-values and compound e-values.

\subsection{Extensions and open questions}

Finally, we briefly mention several directions for further research that are suggested by the findings of this work. 
First, our results show that the PRDS assumption is, for compound p-values, not sufficient for avoiding a $\mathcal{O}(\log m)$ factor in the bound on FDR. Is it possible that an alternative assumption, which allows for a richer class of distributions than independence, would enable a finite-sample bound on FDR that is $\mathcal{O}(\alpha)$?

Second, in the setting of independence, Theorem~\ref{thm:fdr_independent} guarantees that FDR can be  at most $1.93\alpha$, but in practice we might expect that FDR will often be closer to $\alpha$, i.e., without substantial inflation. An interesting question is to determine the types of settings in which substantial inflation is likely, or unlikely, to occur. Moreover, the FDR is the expected value of the false discovery proportion (FDP), but examining its variance and providing bounds on its tail behavior are also important questions. These types of results have been extensively studied for BH in the setting of p-values \citep{ferreira2006benjamini,storey2004strong}; can they be extended to the setting of compound p-values, for tighter control of the FDP?

Next, one drawback of working with the FDR is the possibility of ``cheating with FDR'', in the terminology of \citet{finner2001false}, who explain
    \begin{quote}
        ``...the problematic nature of the FDR concept. If an experimenter is interested in rejecting a hypothesis, he may choose a setting with two hypotheses one of which is known to be rejected with probability near $1$. Then the hypothesis of primary interest can be tested at level $2\alpha$.'' \citep[Section~6]{finner2001false}
    \end{quote}
In other words, since FDR only constrains the proportion of false discoveries, it does not incentivize the analyst to be confident in each individual hypothesis that is rejected. For p-values, this concern motivates the consideration of measures of error such as local FDR \citep{efron2001empirical} rather than FDR. (See \citet{efron2001empirical} and \citet{efron2012large} for Bayesian and empirical Bayesian approaches to the local FDR; additionally, recent work by \citet{soloff2024edge} and \citet{xiang2025frequentist} proposes a method with a finite-sample frequentist guarantee.)
In the setting of compound p-values, this issue with FDR might be magnified, since we have the flexibility to prioritize certain hypotheses even without any non-nulls because the $p_i$'s only need to satisfy an average notion of validity. This motivates an exploration of the possibility of a local FDR based-approach to the multiple testing problem, in the setting of compound p-values.

\subsection*{Acknowledgements}
R.F.B.\ was partially supported by the National Science Foundation via grant DMS-2023109, and by the Office of Naval Research via grant N00014-24-1-2544.  R.J.S.\ was supported by European Research Council Advanced Grant 101019498. 
The authors thank John Duchi, Nikolaos Ignatiadis, Aaditya Ramdas, and Jake Soloff for discussions that helped inspire and shape this work.

\bibliographystyle{apalike}
\bibliography{bib}

\appendix
\setcounter{theorem}{0}
\renewcommand{\thetheorem}{S\arabic{theorem}}
\section{Proofs of upper bounds on FDR}\label{app:proofs_mainresults}

In this section, we prove our main results bounding the FDR
for independent compound p-values (Theorems~\ref{thm:fdr_independent},~\ref{thm:fdr_approx}, and~\ref{thm:fdr_globalnull}).

\subsection{Preliminaries: a reformulation of BH}
We begin by deriving an equivalent characterization of the Benjamini--Hochberg procedure (this result is closely related to leave-one-out style analyses of the BH procedure---see, e.g., \citet{benjamini2001control,ferreira2006benjamini}). For each $i\in\{0,1,\dots,|\nulls|\}$, define
\begin{equation}\label{eqn:reformulate_BH_define_ki}k_i = \max\left\{k\in[m] : i + \sum_{j\not\in\nulls}\One{p_j\leq \frac{\alpha k}{m}} \geq k\right\},\end{equation}
or $k_i=0$ if this set is empty. We have $k_0\leq \dots \leq k_{|\nulls|}$, and $k_i\geq i$ for each $i$, by construction.
Note that $k_i$ is random, as it is a function of the non-null p-values $(p_j)_{j\not\in\nulls}$, but it does not depend on the null p-values. We can interpret $k_i$ as the total number of rejections, if exactly $i$ many true nulls are rejected. Define also
\begin{equation}\label{eqn:reformulate_BH_define_I}I = \max\left\{i\in\{0,1,\dots,|\nulls|\} : \sum_{j\in\nulls}\One{p_j\leq \frac{\alpha k_i}{m}} \geq i\right\}.\end{equation}
The following lemma (which is proved in Appendix~\ref{app:proofs_lemmas}) establishes the connection of these definitions to the BH procedure.
\begin{lemma}\label{lem:BH_condition_on_nonnulls}
    Fix any $p_1,\dots,p_m\in[0,1]$, and let $k_i$ and $I$ be defined as above. Then the BH procedure satisfies 
    \[|\Sbh| = k_I\textnormal{ and }|\Sbh\cap\nulls| = I.\]
\end{lemma}
\noindent In particular, the random variable $I$ defined in~\eqref{eqn:reformulate_BH_define_I} is equal to the number of null hypotheses that are rejected.
We therefore have
\begin{equation}\label{eqn:rewrite_fdr}\fdp(\Sbh) = \frac{|\Sbh\cap\nulls|}{1\vee |\Sbh|} = \frac{I}{1\vee k_I}.\end{equation}

\subsection{Proof of Theorem~\ref{thm:fdr_independent}}
We are now ready to prove our bound on FDR for the setting of independent compound p-values. We will begin with the reformulation of FDR derived above in~\eqref{eqn:rewrite_fdr}: we have
\[\fdp(\Sbh) =\frac{I}{1\vee k_I} = \sum_{i=1}^{|\nulls|} \frac{i}{k_i}\cdot \One{I=i}\]
and so
\[\EEst{\fdp(\Sbh)}{(p_j)_{j\not\in\nulls}} = \sum_{i=1}^{|\nulls|}  \frac{i}{k_i}\cdot q_i\textnormal{ where }q_i=\PPst{I=i}{(p_j)_{j\not\in\nulls}}.\]
A straightforward calculation verifies that this can equivalently be written as
\begin{equation}\label{eqn:rewrite_fdr_qbar_i}\EEst{\fdp(\Sbh)}{(p_j)_{j\not\in\nulls}} = \sum_{i=1}^{|\nulls|}  \frac{i}{k_i}\cdot \bar{q}_i\prod_{j=i+1}^{|\nulls|}(1-\bar{q}_j)\textnormal{ where }\bar{q}_i = \frac{q_i}{1 - \sum_{j=i+1}^{|\nulls|}q_j}.\end{equation}
Note that we can interpret $\bar{q}_i$ as
$\bar{q}_i = \PPst{I=i}{I\leq i; (p_j)_{j\not\in\nulls}}$.
Now  define events
\begin{equation}\label{eqn:define_Ei}\mathcal{E}_i = \left\{ \sum_{j\in\nulls}\One{p_j\leq \frac{\alpha k_i}{m}} \geq i\right\},\end{equation}
and observe that by definition of $I$~\eqref{eqn:reformulate_BH_define_I}, 
\[
\{I=i\} = \bigl\{\mathcal{E}_i \cap \mathcal{E}_{i+1}^c \cap \dots \cap \mathcal{E}_{|\nulls|}^c\bigr\}.
\]
Therefore,
\[\bar{q}_i = \PPst{\mathcal{E}_i}{\mathcal{E}_{i+1}^c \cap \dots \cap \mathcal{E}_{|\nulls|}^c; (p_j)_{j\not\in\nulls}}.\]

Next we will need the following lemma (which is proved in Appendix~\ref{app:proofs_lemmas}):
\begin{lemma}\label{lem:increasing_sets}
    Let $Y_1,\dots,Y_k\in\R$ be independent random variables.
    Then for any increasing sets $A,B\subseteq\R^k$,
    \[\PP{(Y_1,\dots,Y_k)\in A\cap B}\geq \PP{(Y_1,\dots,Y_k)\in A}\cdot \PP{(Y_1,\dots,Y_k)\in B}.\]
\end{lemma}
Applying this result to the independent random variables $(p_j)_{j\in\nulls}$, and noting that $\mathcal{E}_i^c$ and $\mathcal{E}_{i+1}^c \cap \dots \cap \mathcal{E}_{|\nulls|}^c$ (viewed as subsets of $\R^{|\nulls|}$) are both increasing sets,
\[\PPst{\mathcal{E}_i^c \cap \mathcal{E}_{i+1}^c \cap \dots \cap \mathcal{E}_{|\nulls|}^c}{(p_j)_{j\not\in\nulls}}
\geq \PPst{\mathcal{E}_i^c}{(p_j)_{j\not\in\nulls}}\cdot\PPst{\mathcal{E}_{i+1}^c \cap \dots \cap \mathcal{E}_{|\nulls|}^c}{(p_j)_{j\not\in\nulls}}.\]
Therefore,
\begin{align*}
    \bar{q}_i 
    &= \PPst{\mathcal{E}_i}{\mathcal{E}_{i+1}^c \cap \dots \cap \mathcal{E}_{|\nulls|}^c; (p_j)_{j\not\in\nulls}}\\
    &=1-\PPst{\mathcal{E}_i^c}{\mathcal{E}_{i+1}^c \cap \dots \cap \mathcal{E}_{|\nulls|}^c; (p_j)_{j\not\in\nulls}}\\
    &\leq 1 - \PPst{\mathcal{E}_i^c}{ (p_j)_{j\not\in\nulls}}
    =\PPst{\mathcal{E}_i}{ (p_j)_{j\not\in\nulls}}.
\end{align*}

Next, for any integer $i\geq 1$ and any $t\geq 0$, define
\begin{multline}\label{eqn:define_Bi}B_i(t) = \sup\bigg\{ \PP{A_1+ \dots + A_r\geq i} \ : \ r\geq 1, \textnormal{ and }\\\textnormal{$A_1,\dots,A_r$ are independent Bernoulli random variables with $\EE{A_1 + \dots + A_r}\leq t$}\bigg\}.\end{multline}
By taking $r=m$, and defining $A_j = \One{p_j \leq \frac{t}{m}, j\in\nulls}$, we see that
\[\PPst{\sum_{j\in\nulls}\One{p_j\leq \frac{t}{m}}\geq i}{(p_j)_{j\not\in\nulls}} \leq B_i(t),\]
since the $p_j$'s are independent compound p-values.
In particular, by definition of $\mathcal{E}_i$,
\[\PPst{\mathcal{E}_i}{(p_j)_{j\not\in\nulls}} \leq B_i(\alpha k_i).\]
Combining all the work above, we have
\[\EEst{\fdp(\Sbh)}{(p_j)_{j\not\in\nulls}} =  \sum_{i=1}^{|\nulls|}  \frac{i}{k_i}\cdot \bar{q}_i\prod_{j=i+1}^{|\nulls|}(1-\bar{q}_j) 
 \textnormal{ where } 0\leq \bar{q}_i\leq B_i(\alpha k_i).\]
To complete the proof, we need the following lemma (which is proved in Appendix~\ref{app:proofs_lemmas}):
\begin{lemma}\label{lem:t_q_sum}
Let $B_i(t)$ be defined as above. Then for any $L\geq 1$,
\[\sup\left\{ \sum_{i=1}^L  \frac{i}{t_i}\cdot \bar{q}_i\prod_{j=i+1}^L(1-\bar{q}_j) \ : \ t_i>0, \  0\leq \bar{q}_i\leq B_i(t_i), \ \forall i=1,\dots,L\right\} \leq 1.9227.\]
\end{lemma}\noindent
Applying this lemma with $t_i = \alpha k_i$, we therefore have
\[\EEst{\fdp(\Sbh)}{(p_j)_{j\not\in\nulls}} =\alpha \cdot  \sum_{i=1}^{|\nulls|}  \frac{i}{\alpha k_i}\cdot \bar{q}_i\prod_{j=i+1}^{|\nulls|}(1-\bar{q}_j) \leq 1.9227\alpha.\]

\subsection{Proof of Theorem~\ref{thm:fdr_approx}}
Theorem~\ref{thm:fdr_approx} extends the results above to the setting of $(\eps,\delta)$-approximate compound p-values. The proof is very similar to the one above.  First, beginning with Lemma~\ref{lem:BH_condition_on_nonnulls}, we have
\[\frac{|\Sbh\cap\nulls|}{1\vee\big(\frac{m\delta}{\alpha} + |\Sbh|(1+\eps)\big)} = \frac{I}{1 \vee \bigl(\frac{m\delta}{\alpha} + k_I(1+\eps)\bigr)}.\]
(We write $1\vee (\dots)$ in the denominator to simultaneously handle both cases, $\delta=0$ and $\delta>0$.)
Therefore, following the same steps as in the calculation~\eqref{eqn:rewrite_fdr_qbar_i} in the proof of Theorem~\ref{thm:fdr_independent}, we have
\[\EEst{\frac{|\Sbh\cap\nulls|}{1\vee\big(\frac{m\delta}{\alpha} + |\Sbh|(1+\eps)\big)}}{(p_j)_{j\not\in\nulls}} = \sum_{i=1}^{|\nulls|}  \frac{i}{\frac{m\delta}{\alpha} + k_i(1+\eps)}\cdot \bar{q}_i\prod_{j=i+1}^{|\nulls|}(1-\bar{q}_j),\]
where as before, $\bar{q}_i = \PPst{I=i}{I\leq i; (p_j)_{j\not\in\nulls}}$. From the same argument as in the proof of Theorem~\ref{thm:fdr_independent} (i.e., by applying Lemma~\ref{lem:increasing_sets}), we have $\bar{q}_i \leq \PPst{\mathcal{E}_i}{(p_j)_{j\not\in\nulls}}$, where $\mathcal{E}_i$ is defined in~\eqref{eqn:define_Ei}.
Next, using the fact that $p_1,\dots,p_m$ are $(\eps,\delta)$-approximate compound p-values (and are independent), for any $t$ we have
\[\PPst{\sum_{j\in\nulls}\One{p_j\leq \frac{t}{m}}\geq i}{(p_j)_{j\not\in\nulls}} \leq B_i\big(t(1+\eps)+m\delta\big).\] In particular, this means that
\[\bar{q}_i \leq \PPst{\mathcal{E}_i}{(p_j)_{j\not\in\nulls}} \leq B_i\big(\alpha k_i (1+\eps) + m\delta\big),\]
for each $i$.
Combining everything, we have shown that
\[
\EEst{\frac{|\Sbh\cap\nulls|}{1\vee\big(\frac{m\delta}{\alpha} + |\Sbh|(1+\eps)\big)}}{(p_j)_{j\not\in\nulls}} =  \alpha\sum_{i=1}^{|\nulls|}  \frac{i}{m\delta + \alpha k_i(1+\eps)}\cdot \bar{q}_i\prod_{j=i+1}^{|\nulls|}(1-\bar{q}_j)\leq 1.9227\alpha,
\]
where the last step holds by applying Lemma~\ref{lem:t_q_sum}, which completes the proof.

\subsection{Proof of Theorem~\ref{thm:fdr_globalnull}}
We next refine our bound in the special case of the global null, $\nulls=[m]$.  Under the global null, we have $\fdp(\Sbh) = \One{\Sbh\neq\varnothing}$, and moreover, 
\[
\Sbh\neq \varnothing \textnormal{ if and only if $\mathcal{E}_i$ occurs for at least one $i\in[m]$},
\]
or equivalently,
\[\fdr = \PP{\Sbh\neq\varnothing} = \PP{\mathcal{E}_1\cup\dots\cup\mathcal{E}_m}.\]
Here we are defining the events $\mathcal{E}_i$ as before: we have
\[\mathcal{E}_i = \left\{ \sum_{j\in\nulls}\One{p_j\leq \frac{\alpha k_i}{m}} \geq i\right\} = \left\{ \sum_{j=1}^m\One{p_j\leq \frac{\alpha i}{m}} \geq i\right\},\]
since under the global null, we have $\nulls=[m]$, and $k_i = i$ for each $i$.

Write $F_j(t) = \PP{p_j\leq t}$ for the CDF of $p_j$, for each $j\in[m]$, and let $F_j^{-1}(s) = \inf\{t : F_j(t) \geq s\}$ denote its generalized inverse.  We may assume without loss of generality that the $p_j$'s are generated as $p_j = F_j^{-1}(U_j)$, where $U_1,\ldots,U_m \iidsim \textnormal{Unif}[0,1]$, so that
\[
\mathcal{E}_i = \left\{\sum_{j=1}^m \One{F_j^{-1}(U_j)\leq \frac{\alpha i}{m}}\geq i\right\} = \left\{\sum_{j=1}^m \One{U_j\leq F_j\left(\frac{\alpha i}{m}\right)}\geq i\right\}
\]
where the last step holds since, due to the fact that $F_j$ is nondecreasing and right-continuous, this means that $F_j^{-1}(u)\leq t$ if and only if $u\leq F_j(t)$.

Next, let $s_i=(s_{i,1},\dots,s_{i,m})$, where $s_{i,j} = F_j(\frac{\alpha i}{m}) = \PP{p_j\leq \frac{\alpha i}{m}}$, for each $i,j\in[m]$. Under the global null, $\sum_{j=1}^m s_{i,j}\leq \alpha i$, for each $i$, since the $p_j$'s are compound p-values.
Then defining events $\mathcal{U}_i(s_i) = \{\sum_{j=1}^m \One{U_j\leq s_{i,j}} \geq i\}$, since $k_i = i$ under the global null, we therefore have $\mathcal{E}_i = \mathcal{U}_i(s_i)$, for each $i$, and so
\begin{multline*}\fdr = \PP{\mathcal{E}_1\cup\dots\cup\mathcal{E}_m}\\\leq\sup\left\{ \PP{\mathcal{U}_1(s_1)\cup\dots\cup\mathcal{U}_m(s_m)}  \ : \ s_1,\dots,s_m \in[0,1]^m,  \ \sum_{j=1}^m s_{i,j}\leq \alpha i  \ \forall  \ i\right\}.\end{multline*}
The proof is then completed with the following lemma (proved in Appendix~\ref{app:proofs_lemmas}):
\begin{lemma}\label{lem:prob_bound_for_globalnull}
    Fix any $m\geq 1$, and let $U_1,\dots,U_m\iidsim\textnormal{Unif}[0,1]$. Let $s_1,\dots,s_m\in[0,1]^m$ satisfy $\sum_{j=1}^m s_{i,j}\leq \alpha i$ for all $i\in[m]$. Then defining the events $\mathcal{U}_i(s_i)$ as above, 
    \[\PP{\mathcal{U}_1(s_1)\cup\dots\cup\mathcal{U}_m(s_m)}\leq \alpha + 2\alpha^2.\]
\end{lemma}

\section{Proofs of lower bounds on FDR}\label{app:proofs_constructions}
In this section, we verify our constructions establishing
lower bounds on the best possible FDR control results (Propositions~\ref{prop:fdr_example},~\ref{prop:fdr_example_globalnull}, and~\ref{prop:fdr_example_prds}).

\subsection{Proof of Proposition~\ref{prop:fdr_example}}
    Let $k=\frac{1}{2\alpha}$, and define the set of nulls as $\nulls = \{1,\dots,m-3k+2\}\subseteq[m]$. Let $p=(p_1,\dots,p_m)$ have the following distribution: 
    \begin{itemize}
        \item The distribution of the null $p_i$'s is given by
        \[p_1\sim\mathrm{Unif}\left\{\frac{1.5}{m},1\right\}, \quad p_2 = \frac{1}{m}, \quad p_3=\dots=p_{m-3k+2}= 1\]
        (note that, aside from $p_1$, the remaining $m-3k+1$ null $p_i$'s are deterministic).
        \item The non-null $p_i$'s are deterministic, with values
        \[p_{m-3k+3}=\dots=p_{m-k+1} = \frac{1}{m},\quad p_{m-k+2}=\dots=p_m=\frac{1.5}{m}.\]
    \end{itemize}
    The $p_i$'s are mutually independent by definition.
    To see that this construction satisfies the definition of compound p-values, it suffices to observe that
    \[\sum_{i\in\nulls}\PP{p_i\leq \frac{1}{m}} = \PP{p_2\leq \frac{1}{m}} =  1\]
    (since $p_i\leq \frac{1}{m}$ cannot occur for any $i\in\nulls\backslash\{2\}$), and
    \[\sum_{i\in\nulls}\PP{p_i\leq \frac{1.5}{m}}
    = \PP{p_1\leq \frac{1.5}{m}} + \PP{p_2\leq \frac{1.5}{m}} = \frac{1}{2}+1 = 1.5\]
        (since $p_i\leq \frac{1}{m}$ cannot occur for any $i\in\nulls\backslash\{1,2\}$).
        
    Moreover, we can calculate the FDR as follows. With probability $\frac{1}{2}$, we have $p_1 = \frac{1.5}{m}$, and a straightforward calculation shows that on this event,
    \[\Sbh = \left\{ i : p_i \leq \frac{\alpha \cdot 3k}{m}\right\} = \left\{i : p_i \leq \frac{1.5}{m}\right\} =  \{1,2\} \cup \nulls^c,\]
    for a total of $|\Sbh|=3k$ many rejected hypotheses, and so $\fdp(\Sbh) = \frac{2}{3k} = \frac{4}{3}\alpha$.
    Otherwise, with probability $\frac{1}{2}$ we have $p_1 = 1$, and a straightforward calculation shows that on this event,
    \[\Sbh = \left\{i: p_i \leq \frac{\alpha \cdot 2k}{m}\right\} = \left\{i : p_i \leq \frac{1}{m}\right\} = \{2\}\cup \{m-3k+3,\dots,m-k+1\},\]
    for a total of $|\Sbh|=2k$ many rejected hypotheses, and so $\fdp(\Sbh) = \frac{1}{2k} = \alpha$.
    Therefore, $\fdr = \EE{\fdp(\Sbh)} = \frac{1}{2}\cdot \frac{4}{3}\alpha + \frac{1}{2}\cdot \alpha = \frac{7}{6}\alpha$.
    
\subsection{Proof of Proposition~\ref{prop:fdr_example_globalnull}}
    Let $p=(p_1,\dots,p_m)$ have the following distribution: 
    \[p_1 = \begin{cases} \frac{\alpha}{m}, & \textnormal{ with probability $\alpha$},\\
    \frac{2\alpha}{m}, & \textnormal{ with probability $\frac{\alpha}{2}$},\\ 1, & \textnormal{ with probability $1-\frac{3\alpha}{2}$},\end{cases}\]
    and
    \[p_2 = \begin{cases} \frac{2\alpha}{m}, & \textnormal{ with probability $\frac{\alpha}{2}$},\\
    1, & \textnormal{ with probability $1-\frac{\alpha}{2}$,}\end{cases}\]
    with $p_1\independent p_2$, while $p_3 = \dots = p_m = 1$
    (note that, aside from $p_1,p_2$, the remaining $m-2$ $p_i$'s are deterministic).
    The $p_i$'s are mutually independent by definition.
    To see that this construction satisfies the definition of compound p-values, it suffices to observe that
    \[\sum_{i\in\nulls}\PP{p_i\leq \frac{\alpha}{m}} = \PP{p_1\leq \frac{\alpha}{m}} =  \alpha\]
    (since $p_i\leq \frac{\alpha}{m}$ cannot occur for any $i\in[m]\backslash\{1\}$)
    and
    \[\sum_{i\in\nulls}\PP{p_i\leq \frac{2\alpha}{m}} = \PP{p_1\leq \frac{2\alpha}{m}} + \PP{p_2\leq \frac{2\alpha}{m}}  =  \frac{3\alpha}{2} + \frac{\alpha}{2} = 2\alpha\]
        (since $p_i\leq \frac{2\alpha}{m}$ cannot occur for any $i\in[m]\backslash\{1,2\}$).
    Moreover, we can calculate the FDR as follows. The BH procedure rejects at least one hypothesis, i.e.~$\Sbh\neq\varnothing$, whenever
    either $p_1 = \frac{\alpha}{m}$, or $p_1 = p_2 = \frac{2\alpha}{m}$. Therefore,
    \[\fdr = \PP{\Sbh\neq\varnothing} 
    = \PP{p_1 =\frac{\alpha}{m}} + \PP{p_1 = p_2 = \frac{2\alpha}{m}}= \alpha + \frac{\alpha}{2}\cdot \frac{\alpha}{2} = \alpha + \frac{1}{4}\alpha^2,
    \]
    where the first step holds since, under the global null, $\fdp(\Sbh) = \One{\Sbh\neq \varnothing}$.
    
\subsection{Proof of Proposition~\ref{prop:fdr_example_prds}}
    Let $L$ be the largest integer such that $\frac{L(L+1)}{2}\leq m$ (note that $L\approx \sqrt{2m}$), and let $\mathsf{h}_L = 1 + \frac{1}{2}+\dots + \frac{1}{L}$. We can verify that $\mathsf{h}_L\geq \mathsf{h}_m/2$ for any $m\geq 1$. Let $\alpha' = \min\{\alpha, \frac{1}{\mathsf{h}_m}\}$.

    Next we define the distribution of $p=(p_1,\dots,p_m)$. First, we partition $[m]$ into bins, $[m] = B_1 \cup \dots \cup B_L \cup B_*$, where 
    \[B_1 = \{1\}, B_2 = \{2,3\}, B_3 = \{4,5,6\}, \dots,\]
    that is, $|B_\ell| =\ell$ for each $\ell\in[L]$ (and $|B_*| = m - \frac{L(L+1)}{2}$, i.e., the last bin $B^*$ contains any remaining indices). Let $U_1,\dots,U_L\iidsim \textnormal{Unif}[0,1]$, and define
    \[p_i = \frac{\alpha \ell}{m} \cdot \One{U_\ell \leq \frac{\alpha'}{\ell}} + 1 \cdot \One{U_\ell > \frac{\alpha'}{\ell}}\textnormal{ for all $i\in B_\ell$},\]
    for each $\ell\in[L]$, and $p_i=1$ for all $i\in B_*$.
    Note that in this construction, within the $\ell$th bin $B_\ell$, there are $\ell$ many $p_i$, which are all defined to be equal; across the bins, the draws are mutually independent. By construction, this joint distribution satisfies the PRDS assumption~\eqref{eqn:PRDS}.\footnote{We recall from Section~\ref{sec:background_bh} that \citet{guo2008control} derive a construction for verifying that, for arbitrarily dependent p-values, $\alpha\mathsf{h}_m$ is the tightest possible bound on FDR (that is, without a PRDS condition). Readers familiar with that work may note similarities with the construction used here in the proof of Proposition~\ref{prop:fdr_example_prds}. These similarities are interesting since the settings are quite different: as compared to \citet{guo2008control}'s result, here the $p_i$'s are required to be compound p-values (a weaker constraint than requiring p-values), but are assumed to have PRDS dependence (a stronger constraint than arbitrary dependence); these two differences seem to effectively cancel out to allow for the same type of lower bound to hold.} 
    
    To verify that these $p_i$'s are compound p-values, it suffices to calculate, for each $\ell\in[L]$, 
    \begin{multline*}\sum_{i=1}^m\PP{p_i\leq \frac{\alpha \ell}{m}}
    = \sum_{\ell'=1}^L \sum_{i\in B_{\ell'}} \PP{p_i\leq \frac{\alpha \ell}{m}}
    =\sum_{\ell'=1}^L \ell' \cdot \PP{U_{\ell'}\leq \frac{\alpha'}{\ell'}}\cdot \One{\frac{\alpha \ell'}{m} \leq \frac{\alpha \ell}{m}} \\
    = \sum_{\ell'=1}^L \ell' \cdot \frac{\alpha'}{\ell'} \cdot \One{\ell'\leq \ell} = \alpha \ell,\end{multline*}
    where for the next-to-last step, we assume $\alpha>0$ to exclude the trivial case.
    
    Next, to calculate the FDR, we have $\fdr = \PP{\Sbh\neq \varnothing}$ since we are working under the global null, $\nulls=[m]$. Moreover, for any $\ell\in[L]$, if $U_\ell\leq \frac{\alpha'}{\ell}$ then $p_i = \frac{\alpha \ell}{m}$ for all $i\in B_\ell$, i.e., for $\ell$ many $p_i$'s---and therefore, on this event, $\Sbh\supseteq B_\ell$, meaning that each hypothesis corresponding to an index in the $\ell$th bin is rejected. Consequently,
    \[\fdr = \PP{\Sbh\neq \varnothing} \geq \PP{U_\ell\leq \frac{\alpha'}{\ell}\textnormal{ for some $\ell\in[L]$}} = 1 - \prod_{\ell=1}^L \PP{U_\ell > \frac{\alpha'}{\ell}},\]
    where the last step holds since $U_1,\dots,U_L$ are mutually independent. And, we can calculate
    \[\prod_{\ell=1}^L \PP{U_\ell > \frac{\alpha'}{\ell}} = \prod_{\ell=1}^L \left( 1- \frac{\alpha'}{\ell}\right) \leq e^{ - \sum_{\ell=1}^L \frac{\alpha'}{\ell}} = e^{-\alpha' \mathsf{h}_L}
    \leq e^{-\alpha' \mathsf{h}_m/2} \leq 
     1 - \frac{\alpha'\mathsf{h}_m}{2} + \frac{(\alpha'\mathsf{h}_m)^2}{8},
    \]
    where the next-to-last step holds since $\mathsf{h}_L\geq \mathsf{h}_m/2$.
    Therefore,
    \[\fdr = \EE{\fdp(\Sbh)} \geq \frac{\alpha'\mathsf{h}_m}{2} - \frac{(\alpha'\mathsf{h}_m)^2}{8},\]
    which completes the proof by our definition of $\alpha'=\min\{\alpha,\frac{1}{\mathsf{h}_m}\}$.

\section{Proofs for examples}\label{app:proofs_examples}
In this section we present proofs for all of the results stated in Section~\ref{sec:examples}, for the various examples of settings where compound p-values arise.

\subsection{Proofs for Example~\ref{example:empirical_Bayes}: empirical Bayes}

\begin{proof}[Proof of Proposition~\ref{prop:average_null_CDFs}]
To prove the first claim, define $p(y) = \frac{1}{m}\sum_{j\in\nulls} F_j(y)$, so that $p_i \geq p(X_i)$ almost surely for each $i\in\nulls$. Note that $y\mapsto p(y)$ is nondecreasing (since each $F_j$ is a CDF and is therefore nondecreasing).

For a fixed $t \in [0,1]$, let $y^* = \sup\{y:p(y)\leq t\}$. We now split into cases. First, if $p(y^*)\leq t$, then 
\[p(y)\leq t \ \Longleftrightarrow \ y\leq y^*,\] and so
\[\sum_{i\in\nulls} \PP{p_i\leq t} \leq \sum_{i\in\nulls}\PP{p(X_i)\leq t}= 
\sum_{i\in\nulls} \PP{X_i\leq y^*} = 
\sum_{i\in\nulls} F_i(y^*)= m\cdot p(y^*) \leq mt.\]
If instead $p(y^*)>t$, then  
\[p(y)\leq t \ \Longleftrightarrow \ y< y^*,\] and so
\begin{multline*}\sum_{i\in\nulls} \PP{p_i\leq t} \leq\sum_{i\in\nulls}\PP{p(X_i)\leq t}= 
\sum_{i\in\nulls} \PP{X_i < y^*} = 
\sum_{i\in\nulls} \lim_{y\nearrow y^*} \PP{X_i \leq y} \\= 
\sum_{i\in\nulls} \lim_{y\nearrow y^*} F_i(y) = m\cdot \lim_{y\nearrow y^*}  p(y) \leq mt.\end{multline*}
The second claim (with $\bar{F}_j$'s in place of $F_j$'s) is proved analogously.
\end{proof}

\subsection{Illustration and proofs for Example~\ref{example:decreasing_densities}: decreasing densities}\label{app:decreasing_density_appendix}

First, we illustrate the construction of the approximate compound p-values $p_i$, with a simulated example. We set $m=10{,}000$, with $|\nulls|=9{,}000$ true nulls. We draw $X_i \sim |\mathcal{N}(0,1)|$ for $i\in\nulls$ (i.e., the absolute value of a draw from the standard normal distribution), and $X_i \sim |\mathcal{N}(10,2)|$ for $i\not\in\nulls$, independently. Figure~\ref{fig:decreasing_density_sim} shows the resulting values $p_i$, calculated as in~\eqref{eqn:define_p_i_decreasing_densities}. We observe that there are a large number of very small p-values, corresponding to the $X_i$ values not drawn from the null.

\begin{figure}[t]
    \includegraphics[width=\textwidth]{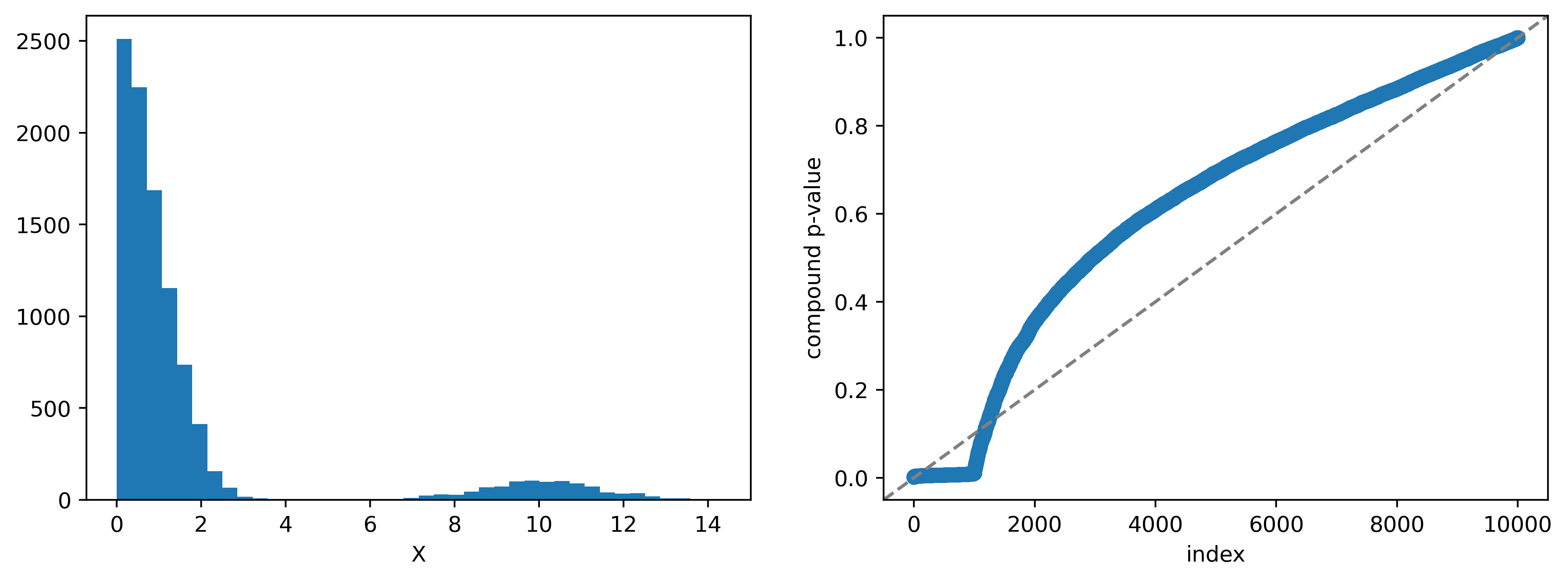}
    \caption{An illustration of the approximate compound p-values constructed for Example~\ref{example:decreasing_densities} (decreasing densities). The left panel shows a histogram of the observed data values $X_i$. The right panel shows the values $p_1,\dots,p_m$ sorted in increasing order (with a dashed line indicating the diagonal---i.e., uniformly distributed p-values would cluster near the dashed line). See Appendix~\ref{app:decreasing_density_appendix} for details.}
    \label{fig:decreasing_density_sim}
\end{figure}

We next present the proofs of the results for this example.
\begin{proof}[Proof of Proposition~\ref{prop:decreasing_density}]
    First, for each $j\in\nulls$, let $\bar{F}_j$ be the right-tail CDF of the distribution $P_j$, i.e., $\bar{F}_j(x)=\PP{X_j\geq x}$. Define
    \[p^*_i = g^*(X_i)\textnormal{ where }g^*(x) = \frac{1}{m}\sum_{j\in\nulls}\bar{F}_j(x).\]
    By Proposition~\ref{prop:average_null_CDFs}, $p^*_1,\dots,p^*_m$ are compound p-values. 

    Next, observe that on the event that $g^*\in\mathcal{G}(\Delta)$, for each $i\in[m]$ we have
    \[p_i \geq p^*_i,\]
    and consequently, for any $t\in[0,1]$,
    \begin{align*}
        \sum_{i\in\nulls}\PP{p_i\leq t}
        &\leq \sum_{i\in\nulls}\left(\PP{p^*_i\leq t} + \PP{g^*\not\in\mathcal{G}(\Delta)}\right)\\
        &\leq m\left(t + \PP{g^*\not\in\mathcal{G}(\Delta)}  \right),
    \end{align*}
    where for the last step we use the fact that $p^*_1,\dots,p^*_m$ are compound p-values.
    Finally, $g^*$ is a convex and nonincreasing function (since each $\bar{F}_j$ is a right-tail CDF of a distribution with nonincreasing density, and is therefore itself convex and nonincreasing). We then have $g^*\in\mathcal{G}(\Delta)$ if and only if 
    $\max_{i=0,\dots,m} \{g^*(X_{(i)}) - g^*(X_{(i+1)})\}\leq \Delta$. Applying Lemma~\ref{lem:largest_gap_} verifies that this holds with probability $\geq 1-\frac{1}{m}$, which completes the proof.
\end{proof}

\begin{lemma}\label{lem:largest_gap_}
    Under the notation and assumptions of Proposition~\ref{prop:decreasing_density}, 
    \[\PP{\max_{i=0,\dots,m} \{g^*(X_{(i)}) - g^*(X_{(i+1)})\} \leq \Delta}\geq 1- \frac{1}{m}.\]
\end{lemma}
\begin{proof}[Proof of Lemma~\ref{lem:largest_gap_}]
    First, since each $\bar{F}_j$ is the right-tail CDF of a continuous distribution, $g^*:[0,\infty)\to[0,1]$ is a continuous and nonincreasing function, with $g^*(0)=\frac{|\nulls|}{m}$ and $\lim_{x\to\infty}g^*(x)=0$. Next, let $M=\sup\{x:g^*(x)>0\}$ (note that we may have $M=+\infty$). For each $j\in\nulls$, since $P_j$ has a nonincreasing density, this means that $\bar{F}_j$ is nonincreasing and convex, and therefore its derivative $\bar{F}'_j$ is nonpositive and nondecreasing. Therefore, the derivative $g^*{}'(x) = \frac{1}{m}\sum_{j\in\nulls}\bar{F}_j'(x)$ is a nonpositive and nondecreasing function, and so $g^*$ is strictly decreasing on $[0,M)$. From this point on we will view $g^*$ as a function on this domain, $g^*:[0,M)\to(0,\frac{|\nulls|}{m}]$, and will write $g^*{}^{-1}:(0,\frac{|\nulls|}{m}]\to[0,M)$ as its inverse, which is therefore also strictly decreasing.

    Now let $U\sim \textnormal{Unif}\left[0,\frac{|\nulls|-2\log m}{m}\right]$ be a uniform random variable, independent of $X_1,\ldots,X_m$. (Note that we can assume $\Delta\leq \frac{|\nulls|}{m}$, since otherwise the result is trivial, and so $|\nulls|-2\log m = |\nulls| - m\Delta + 1>0$.)
    Define the event $\mathcal{E}$ as
    \[\mathcal{E} = \left\{ \sum_{j=1}^m\One{g^*(X_j) \in\left(U,U+\frac{2\log m}{m}\right)} = 0\right\}.\] 
    Note that, for each $j\in\nulls$,
    \begin{multline*}\PPst{g^*(X_j) \in\left(U,U+\frac{2\log m}{m}\right)}{U}
    = \PPst{g^*{}^{-1}\left(U+\frac{2\log m}{m}\right) < X_j < g^*{}^{-1}(U)}{U}\\
    = \bar{F}_j\left(g^*{}^{-1}\left(U+\frac{2\log m}{m}\right)\right) - \bar{F}_j(g^*{}^{-1}(U)).\end{multline*}
    Therefore,
    \begin{align*}
        \PPst{\mathcal{E}}{U}
        &\leq \PPst{\sum_{j\in\nulls}\One{g^*(X_j) \in\left(U,U+\frac{2\log m}{m}\right)}=0}{U}\\
        &= \prod_{j\in\nulls}\PPst{g^*(X_j) \not\in\left(U,U+\frac{2\log m}{m}\right)}{U}\textnormal{ by independence of $X_1,\dots,X_m,U$}\\
        &= \prod_{j\in\nulls} \left(1 -  \bar{F}_j\left(g^*{}^{-1}\left(U+\frac{2\log m}{m}\right)\right) - \bar{F}_j(g^*{}^{-1}(U))\right) \\
        &\leq \exp\left\{ - \sum_{j\in\nulls} \left(\bar{F}_j\left(g^*{}^{-1}\left(U+\frac{2\log m}{m}\right)\right) - \bar{F}_j(g^*{}^{-1}(U))\right)\right\}\\
        &=\exp\left\{- m g^*\left(g^*{}^{-1}\left(U+\frac{2\log m}{m}\right)\right) + mg^*(g^*{}^{-1}(U))\right\} 
         = e^{-m \cdot \frac{2\log m}{m}} = \frac{1}{m^2},
    \end{align*}
    where the last line holds since, by definition of $g^*$,
    \[\sum_{j\in\nulls} \bar{F}_j(g^*{}^{-1}(t)) = mg^*(g^*{}^{-1}(t)) = mt\]
    for any $t\in(0,\frac{|\nulls|}{m}]$. Therefore, after marginalizing over $U$, we have $\PP{\mathcal{E}}\leq \frac{1}{m^2}$.
    
    Next, after observing the data $X_1,\dots,X_m$, suppose that $g^*(X_{(i)}) - g^*(X_{(i+1)})>\Delta$ for some $i$. Then conditional on the data,     
    \begin{multline*}\PPst{U\in \left[g^*(X_{(i+1)}),g^*(X_{(i)})-\frac{2\log m}{m}\right]}{X_1,\dots,X_m} \\ = \frac{\left(g^*(X_{(i)})-\frac{2\log m}{m}\right) - g^*(X_{(i+1)}) }{\frac{|\nulls|-2\log m}{m}} > \frac{\Delta - \frac{2\log m}{m}}{\frac{|\nulls|-2\log m}{m}} = \frac{1}{|\nulls|-2\log m} \geq \frac{1}{m},\end{multline*}
    by definition of $\Delta = \frac{\log(em^2)}{m} = \frac{1 + 2\log m}{m}$.
    And,
    if the random variable $U$ falls into the interval $[g^*(X_{(i+1)}),g^*(X_{(i)})-\frac{2\log m}{m}]$, then since no data points can lie between $X_{(i)}$ and $X_{(i+1)}$, for each $j\in[m]$ we have
    \[X_j \not \in (X_{(i)},X_{(i+1)}) \ \Longrightarrow \ g^*(X_j) \not \in (g^*(X_{(i+1)}),g^*(X_{(i)})) \ \Longrightarrow \ g^*(X_j) \not \in \left(U,U+\frac{2\log m}{m}\right),\]
    and therefore the event $\mathcal{E}$ occurs. We have thus shown that
    \[\PPst{\mathcal{E}}{X_1,\dots,X_m} \geq \frac{1}{m} \cdot \One{\max_{i\in\{0,\dots,m\}}\{g^*(X_{(i)}) - g^*(X_{(i+1)})\}>\Delta}.\]
    Combining our calculations, then,
    \[\frac{1}{m}\cdot \PP{\max_{i\in\{0,\dots,m\}}\{g^*(X_{(i)}) - g^*(X_{(i+1)})\}>\Delta} \leq \PP{\mathcal{E}}\leq \frac{1}{m^2},\]
    which completes the proof.
\end{proof}

\begin{proof}[Proof of Proposition~\ref{prop:decreasing_density_compute}]
    
    Fix any $g\in\mathcal{G}(\Delta)$.
    First we will show that
    \begin{equation}\label{eqn:leq_prop:decreasing_density_compute}g(X_{(i)})\leq \sum_{j=i}^m w_j\end{equation} for all $i=1,\dots,m$, by induction. At $i=m$, we have
    \[g(X_{(m)}) = g(X_{(m)}) - g(+\infty) = g(X_{(m)})-g(X_{(m+1)}) \leq \Delta = w_m,\]
    where inequality holds by definition of $\mathcal{G}(\Delta)$. 
    
    Next fix any $i<m$ and assume that the claim~\eqref{eqn:leq_prop:decreasing_density_compute} holds with $i+1$ in place of $i$. We will split into cases. First, suppose that $w_i = \frac{1-\sum_{j=i+1}^m w_j}{\frac{X_{(i+1)}}{X_{(i+1)} - X_{(i)}}}$. By convexity of $g$, it holds that
    \[\frac{g(X_{(i)}) - g(X_{(i+1)})}{X_{(i+1)}-X_{(i)}} \leq \frac{g(X_{(0)}) - g(X_{(i+1)})}{X_{(i+1)}-X_{(0)}} = \frac{g(X_{(0)}) - g(X_{(i+1)})}{X_{(i+1)}} \leq \frac{1 - g(X_{(i+1)})}{X_{(i+1)}},\]
    and so
    \begin{align*}
        g(X_{(i)})
        &\leq g(X_{(i+1)}) + \left(1 - g(X_{(i+1)})\right)\cdot \frac{X_{(i+1)}-X_{(i)}}{X_{(i+1)}}\\
        &= g(X_{(i+1)})\cdot \frac{X_{(i)}}{X_{(i+1)}} +  \frac{X_{(i+1)}-X_{(i)}}{X_{(i+1)}}\\
        &\leq\sum_{j=i+1}^m w_j\cdot \frac{X_{(i)}}{X_{(i+1)}} +  \frac{X_{(i+1)}-X_{(i)}}{X_{(i+1)}}\textnormal{\quad by induction}\\
        &=\sum_{j=i+1}^m w_j + \left( 1- \sum_{j=i+1}^m w_j\right)\cdot \frac{X_{(i+1)}-X_{(i)}}{X_{(i+1)}}\\
        &=\sum_{j=i}^m w_j.
    \end{align*}
    
    Next, we instead consider the case that $w_i = \frac{\Delta}{\frac{X_{(j+1)} - X_{(j)}}{X_{(i+1)} - X_{(i)}}}$ for some $j\in\{0,\dots,i\}$. 
    We write
    \[g(X_{(i)}) = g(X_{(i)}) - g(X_{(i+1)}) + g(X_{(i+1)}) \leq g(X_{(i)}) - g(X_{(i+1)}) + \sum_{j=i+1}^m w_j, \]
    where the  inequality holds by induction. And, by convexity of $g$, since $j\leq i$ it holds that
    \[\frac{g(X_{(i)}) - g(X_{(i+1)})}{X_{(i+1)}-X_{(i)}} \leq \frac{g(X_{(j)}) - g(X_{(j+1)})}{X_{(j+1)}-X_{(j)}} \leq \frac{\Delta}{X_{(j+1)}-X_{(j)}},\]
    where the last step holds by definition of $\mathcal{G}(\Delta)$. Consequently,
    \[g(X_{(i)}) \leq \frac{\Delta}{\frac{X_{(j+1)} - X_{(j)}}{X_{(i+1)} - X_{(i)}}}+ \sum_{j=i+1}^m w_j = \sum_{j=i}^m w_j. \]
    We have therefore shown that~\eqref{eqn:leq_prop:decreasing_density_compute} holds for all $i\in[m]$ and for any $g\in\mathcal{G}(\Delta)$, and consequently,
    \[\sup_{g\in\mathcal{G}(\Delta)}g(X_{(i)})\leq \sum_{j=i}^m w_j\textnormal{ for all $i\in[m]$}.\]

    Now we prove the converse. 
    Define a function $g:[0,\infty)\to\R$, as follows: the function will interpolate linearly between the points
    \[(0, w_0 + \dots + w_m), \ (X_{(1)}, w_1 + \dots + w_m), \ \dots, \ (X_{(m)}, w_m),\]
    and then we will set $g(x)=w_m$ for all $x\geq X_{(m)}$. We will now show that $g\in\mathcal{G}(\Delta)$. We have $w_i\geq 0$ and $\sum_{i=0}^m w_i \leq 1$ by definition of the $w_i$'s (specifically, this is true by induction---for each $i$, $w_i \leq \frac{1 - \sum_{j=i+1}^m w_j}{\frac{X_{(i+1)}}{X_{(i+1)}-X_{(i)}}}\leq 1 - \sum_{j=i+1}^m w_j$). Therefore we see that $g$ takes values in $[0,1]$ and is nonincreasing, by construction. And, for each $i=0,\dots,m$,
    \[g(X_{(i)}) - g(X_{(i+1)}) = w_i \leq \Delta.\]
    Finally, to verify convexity of $g$, by construction of $g$ we only need to check that
    \[\frac{w_i}{X_{(i+1)}-X_{(i)}} \leq \frac{w_{i-1}}{X_{(i)}-X_{(i-1)}}\]
    for each $i=1,\dots,m-1$.  We again split into two cases. If $w_{i-1}=\frac{1-\sum_{j=i}^m w_j}{ \frac{X_{(i)}}{X_{(i)} - X_{(i-1)}}} $, then 
    \begin{align*}
        \frac{w_{i-1}}{X_{(i)}-X_{(i-1)}}
        &= \frac{1-\sum_{j=i}^m w_j}{X_{(i)}}\\
        &= \frac{1-\sum_{j=i+1}^m w_j - w_i}{X_{(i)}}\\
        &\geq \frac{1-\sum_{j=i+1}^m w_j - \frac{1-\sum_{j=i+1}^m w_j}{\frac{X_{(i+1)}}{X_{(i+1)}-X_{(i)}}}}{X_{(i)}}\\
        &= \frac{\left(1-\sum_{j=i+1}^m w_j\right) \cdot \left( 1- \frac{1}{\frac{X_{(i+1)}}{X_{(i+1)}-X_{(i)}}}\right)}{X_{(i)}}\\
        &=\frac{1-\sum_{j=i+1}^m w_j}{X_{(i+1)}}
        \geq \frac{w_i}{X_{(i+1)}-X_{(i)}},
    \end{align*}
    where both inequalities hold by definition of $w_i$.
    If not, then instead we must have $w_{i-1} = \frac{\Delta}{\frac{X_{(j+1)}-X_{(j)}}{X_{(i)}-X_{(i-1)}}}$ for some $j\leq i-1$. Then
    \[\frac{w_{i-1}}{X_{(i)}-X_{(i-1)}} = \frac{\Delta}{X_{(j+1)}-X_{(j)}} = \frac{1}{X_{(i+1)} - X_{(i)}}\cdot \frac{\Delta}{ \frac{X_{(j+1)} - X_{(j)}}{X_{(i+1)} - X_{(i)}}} \geq \frac{1}{X_{(i+1)} - X_{(i)}}\cdot w_i,\]
    where again the last step holds by definition of $w_i$.
\end{proof}

\subsection{Proofs for Example~\ref{example:permutation_tests}: permutation tests}

\begin{proof}[Proof of Proposition~\ref{prop:perm_test_compound}]
Fix any $i\in\nulls$. Since $X_{i1},\dots,X_{in_i}$ are i.i.d., and are independent of $\{X_j\}_{j\in[m]\backslash\{i\}}$, it holds that $X_{i1},\dots,X_{in_i}$ are i.i.d.\ conditionally on $\{\widehat{P}_j\}_{j\in[m]\backslash\{i\}}$. Moreover, since $\widehat{P}_i$ is a symmetric function of the $n_i$ random variables $X_{i1},\dots,X_{in_i}$ (that is, the empirical distribution $\widehat{P}_i$ is invariant to permuting these $n_i$ many values), it holds that $X_{i1},\dots,X_{in_i}$ are exchangeable conditionally on $\widehat{P}_i$ (and on $\{\widehat{P}_j\}_{j\in[m]\backslash \{i\}}$). In particular, we have
\[X_i \mid \{\widehat{P}_j\}_{j\in[m]} \ \eqd X_i^\sigma \mid \{\widehat{P}_j\}_{j\in[m]},\]
for each $i\in\nulls$ and each $\sigma\in\mathcal{S}_{n_i}$. This implies that
$X_i \mid \{\widehat{P}_j\}_{j\in[m]} \ \eqd X_i^\sigma \mid \{\widehat{P}_j\}_{j\in[m]}$
also holds when $\sigma$ is a \emph{random} permutation, sampled uniformly from $\sigma\in\mathcal{S}_{n_i}$.
Consequently, the function
\[\bar{F}_i(t)= \frac{1}{n_i!}\sum_{\sigma\in\mathcal{S}_{n_i}}\One{T_i(X^\sigma_i)\geq t}\]
is equal to the right-tailed (conditional) CDF of the null distribution of $T_i(X_i)$, i.e., 
\[\bar{F}_i(t) = \PPst{T_i(X_i) \geq t}{\{\widehat{P}_j\}_{j\in[m]}}.\]
Observe that
\[p_i = \frac{1}{m}\sum_{j=1}^m \bar{F}_j(T_i(X_i)),\]
by construction.
It then follows immediately from Proposition~\ref{prop:average_null_CDFs} (applied after conditioning on $\{\widehat{P}_j\}_{j\in[m]}$, and with $T_i(X_i)$ in place of $X_i$) that $p_1,\dots,p_m$ are independent compound p-values, conditionally on $\{\widehat{P}_j\}_{j\in[m]}$.
Finally, we prove the marginal claim: we have
\[\sum_{i\in\nulls}\PP{p_i\leq t}  = \EE{\sum_{i\in\nulls}\PPst{p_i\leq t}{\{\widehat{P}_j\}_{j\in[m]}}} \leq \EE{mt} = mt,\]
where the inequality holds since $p_1,\dots,p_m$ are compound p-values conditionally on $\{\widehat{P}_j\}_{j\in[m]}$.
\end{proof}

\subsection{Proofs for Example~\ref{example:MC_pvalues}: Monte Carlo p-values}

\begin{proof}[Proof of Proposition~\ref{prop:MC_compound}]
This proof is very similar to the proof of Proposition~\ref{prop:perm_test_compound}.
Fix any $i\in\nulls$. By an identical argument, it holds that $X_i,X_{i1},\dots,X_{iK}$ are exchangeable conditionally on $\{\widehat{P}_j\}_{j\in[m]}$. Writing $X_{i0}=X_i$, it therefore holds that 
\[X_i \mid\{\widehat{P}_j\}_{j\in[m]} \ \eqd \ X_{iL} \mid \{\widehat{P}_j\}_{j\in[m]}, \]
where $L\sim\textnormal{Unif}(\{0,1,\dots,K\})$, and therefore
\[\bar{F}_i(t)= \frac{1}{1+K}\left\{ \One{T_i(X_i;\widehat{P}_i)\geq t} + \sum_{k=1}^K \One{T_i(X_{ik};\widehat{P}_i)\geq t}\right\}\]
is equal to the right-tailed (conditional) CDF of the null distribution of $T_i(X_i;\widehat{P}_i)$. We can also verify that
\[p_i = \frac{1}{m}\sum_{j=1}^m \bar{F}_j(T_i(X_i;\widehat{P}_i)),\]
and the result then follows by applying Proposition~\ref{prop:average_null_CDFs} as in the proof of Proposition~\ref{prop:perm_test_compound}.
\end{proof}

\section{Additional examples}
\label{app:additional_examples}

In this section we present several additional examples of settings where compound p-values may arise. 

\refstepcounter{examplesettings}\label{example:data_alignment}
\subsection{Example \theexamplesettings: data alignment}
Consider the problem of uncertain alignment between different datasets: for instance, given access to two databases, we would like to identify entries that correspond to the same individual, and may want to flag entries in one database that do not appear to correspond to any individual in the other.

To fix a concrete version of this question, suppose we have the following:
\begin{itemize}
    \item In the first dataset, we have unlabeled data: a feature vector $X_i\in\R^d$ for each individual $i\in[m]$.
    \item In the second dataset, we have labeled data points of the form $(\tilde{X}_i,Y_i)$, where $\tilde{X}_i$ is a noisy version of the individual's feature vector, and $Y_i$ is a label.
\end{itemize}
For example, if the label $Y_i$ records sensitive information, then $\tilde{X}_i$ might be a privatized version of the true feature vector for that individual, so that they cannot be easily identified in the sensitive labeled dataset.
Naturally, for the sake of privacy, we do not know how the entries of the two datasets correspond to each other---that is, $\tilde{X}_i$ is a noisy version of $X_j$, for some other index $j$ not necessarily equal to $i$.
This framework is related to the problem of regression with an unknown permutation, studied by, e.g., \citet{pananjady2017linear,unnikrishnan2018unlabeled,pananjady2022isotonic}.

Suppose that we now would like to identify whether any entries of the labeled dataset do not correspond to any individual in the unlabeled dataset. We define the $i$th null hypothesis as
\[H_{0,i}: \ \tilde{X}_i  = X_{\pi(i)} + \omega_i,\] where $\omega_i \iidsim P$ for a known distribution $P$, i.e., the features are privatized by adding noise drawn from $P$, independently of the $X_i$'s. Here $\pi$ is a fixed and unknown permutation that specifies the correspondence between the individuals in the two datasets. In other words, rejecting $H_{0,i}$ means that we conclude the entry $(\tilde{X}_i,Y_i)$ in the labeled dataset does not correspond to any individual in the unlabeled dataset.
Since $\pi$ is unknown, we do not have enough information to construct a p-value for testing an individual data entry $\tilde{X}_i$. Instead, we can construct compound p-values by considering the dataset as a whole.

\begin{proposition}\label{prop:alignment_compound}
Under the setting and assumptions above, let $T:\R^d\to\R$ be any function, and define
\[p_i = \frac{1}{m}\sum_{j=1}^m \bar{F}_j(T(\tilde{X}_i))\]
where
\[\bar{F}_j(t) = \Pp{\omega\sim P}{T(X_j + \omega) \geq t}.\]
Then $p_1,\dots,p_m$ are compound p-values.
\end{proposition}
\begin{proof}[Proof of Proposition~\ref{prop:alignment_compound}]
First, observe that we can equivalently define
\[p_i = \frac{1}{m}\sum_{j=1}^m \bar{F}_{\pi(j)}(T(\tilde{X}_i)),\]
where $\pi$ is the unknown permutation.
    The result then follows immediately from Proposition~\ref{prop:average_null_CDFs}, since $\bar{F}_{\pi(i)}$ is the right-tailed CDF of the null distribution of $T(\tilde{X}_i) = T(X_{\pi(i)}+\omega_i)$, for each $i\in[m]$, according to our definition of the null hypothesis $H_{0,i}$.
\end{proof}

\refstepcounter{examplesettings}\label{example:weighted_pvalues}
\subsection{Example \theexamplesettings: weighted p-values}
Let $p^*_1,\dots,p^*_m$ be p-values, and let $w_1,\dots,w_m\geq 0$ be fixed weights satisfying $\sum_{i=1}^m w_i=m$. 
\citet{genovese2006false} define ``weighted p-values'',
\[p_i = \min\left\{ \frac{p^*_i}{w_i}, 1\right\}.\]
While these are no longer p-values in the sense of~\eqref{eqn:valid_pvalues} (unless we choose $w_i\equiv 1$), introducing weights is useful as it allows for the analyst to incorporate prior information regarding the different hypotheses $i\in[m]$, or to give different priority to these hypotheses (choosing a larger $w_i$ means that it will be easier to reject $H_{0,i}$).

\citet{ignatiadis2024compound} observe that the $p_i$'s are an example of compound p-values.
Interestingly, under the assumption that the $p^*_i$'s are independent, \citet{genovese2006false} prove that the BH procedure applied to the $p_i$'s maintains FDR control at level $\alpha$ (i.e., there is no inflation of the FDR); the same holds if the $p^*_i$'s satisfy the PRDS condition \citep{ramdas2019unified}. More generally, weighted p-values may be constructed by choosing data-dependent $w_i$, as long as the $w_i$'s are (compound) e-values and are independent of the $p_i$'s \citep{ignatiadis2024values}.

\refstepcounter{examplesettings}\label{example:Gaussian_means}
\subsection{Example \theexamplesettings: Gaussian means}
Our final example builds on one developed by \citet{ignatiadis2025empirical}. Suppose that for each $i\in[m]$, we observe data
\[Y_{i1},\dots,Y_{in}\iidsim \mathcal{N}(\mu_i,\sigma^2_i),\]
with the aim of testing the null hypothesis that $\mu_i=0$. If the variances $\sigma^2_i$ are known, then we can simply compute z-statistics $\frac{\bar{Y}_i}{\sigma_i/\sqrt{n}}$ (where $\bar{Y}_i$ denotes the mean of the $i$th sample), and the compare against the standard normal distribution to produce a p-value for each null. If the $\sigma^2_i$'s are unknown nuisance parameters, then we can instead run t-tests on the t-statistics $\frac{\bar{Y}_i}{S_i/\sqrt{n}}$, where $S_i^2$ is the $i$th sample variance---but if the sample size $n$ for each test is small, the high variance of $S_i$ may lead to noisy p-values and consequently low power. 

We will now develop a construction of compound p-values for this problem, in the setting where the variances $\sigma^2_i$ are unknown.\footnote{
\citet{ignatiadis2025empirical} also propose a construction of compound p-values for this problem, but work in the setting where the analyst has access to $\frac{1}{m}\sum_{i=1}^m\delta_{\sigma^2_i}$, the empirical distribution of the variance parameters. Their results construct compound p-values, using a different construction from~\eqref{eqn:p_star_gaussian_means}---their proposal \citep[Theorem 21]{ignatiadis2025empirical} also uses information from the sample variances $S^2_i$ for a more instance-adaptive approach. They also develop an asymptotically valid approach \citep[Proposition 15${}^*$]{ignatiadis2025empirical} that does not require knowledge of $\frac{1}{m}\sum_{i=1}^m \delta_{\sigma^2_i}$,  and establish an asymptotic bound on FDR, for this setting. See also \citet{ignatiadis2024compound} for a construction of compound e-values, rather than p-values, in this setting.}
To develop our solution, first suppose we had knowledge of the variances $\sigma^2_i$.
Then by Proposition~\ref{prop:average_null_CDFs}, 
\begin{equation}\label{eqn:p_star_gaussian_means}p^*_i = \frac{1}{m}\sum_{j=1}^m 2\bar{\Phi}\left(\frac{|\bar{Y}_i|}{\sigma_j/\sqrt{n}}\right)\end{equation}
define compound p-values, where $\bar{\Phi}(x)=1-\Phi(x)$ is the right-tailed CDF of the standard normal distribution.
Now, since $\sigma_j$'s are unknown, we will need to approximate the $p^*_i$'s, for which we will use the following lemma:
\begin{lemma}\label{lem:Phi_chisq}
    Let $n\geq 3$ and let $X\sim\sigma^2\chi^2_{n-1}$. Then for any $y>0$,
    \[\EE{1 - F_{\textnormal{Beta}(\frac{1}{2},\frac{n}{2}-1)}\left(\frac{y^2}{X}\right)} = 2\bar{\Phi}\left(\frac{y}{\sigma}\right),\]
    where $F_{\textnormal{Beta}(a,b)}$ is the CDF of the $\textnormal{Beta}(a,b)$ distribution.
\end{lemma}
Noting that $(n-1)S^2_j\sim\sigma^2_j\chi^2_{n-1}$ for each $j$, this motivates the following approximation to~\eqref{eqn:p_star_gaussian_means}:
\[p_i =\frac{1}{m}\sum_{j=1}^m\left(1 - F_{\textnormal{Beta}(\frac{1}{2},\frac{n}{2}-1)}\left(\frac{n\bar{Y}_i^2}{(n-1)S^2_j}\right)\right).\]

\begin{proposition}\label{prop:Gaussian_means_approx_compound}
    Let $Y_{ij}\sim \mathcal{N}(\mu_i,\sigma^2_i)$, independently for each $i\in[m]$ and $j\in[n]$, where the $i$th null hypothesis is given by \[H_{0,i}:\ \mu_i=0.\]
    Then $p_1,\dots,p_m$ (as defined above) are $(0,\frac{1}{m})$-approximate compound p-values.
\end{proposition}

\begin{proof}[Proof of Lemma~\ref{lem:Phi_chisq}]
    Let $v\in\R^{n-1}$ be any unit vector, and let $U\in\R^{n-1}$ be a unit vector drawn uniformly from the sphere. Then 
    \[(U^\top v)^2 \sim \textnormal{Beta}\left(\frac{1}{2},\frac{n}{2}-1\right).\]
    Now let $Z\sim \mathcal{N}(0,\sigma^2\mathbf{I}_{n-1})$, then by taking $v=Z/\|Z\|_2$,
    \[\frac{(U^\top Z)^2}{\|Z\|^2_2} \mid Z \sim  \textnormal{Beta}\left(\frac{1}{2},\frac{n}{2}-1\right).\]
    Therefore,
    \[1 - F_{\textnormal{Beta}(\frac{1}{2},\frac{n}{2}-1)}\left(\frac{y^2}{\|Z\|^2_2}\right) = \PPst{\frac{(U^\top Z)^2}{\|Z\|^2_2} \geq \frac{y^2}{\|Z\|^2_2}}{Z} = \PPst{(U^\top Z)^2 \geq y^2}{Z} .\]
    Taking the expected value, we then have
    \[\EE{1 - F_{\textnormal{Beta}(\frac{1}{2},\frac{n}{2}-1)}\left(\frac{y^2}{\|Z\|^2_2}\right)} = \EE{\PPst{(U^\top Z)^2 \geq y^2}{Z} } = \PP{(U^\top Z)^2\geq y^2}.\]
    But marginally we have $(U^\top Z)^2\sim\sigma^2\chi^2_1$, and so
    \[\PP{(U^\top Z)^2\geq y^2} = 2\bar{\Phi}(y/\sigma),\]
    as desired. Since $\|Z\|^2_2\eqd X$ this completes the proof.
\end{proof}
\begin{proof}[Proof of Proposition~\ref{prop:Gaussian_means_approx_compound}]
If $H_{0,i}$ is true, then note that 
\[n\bar{Y}_i^2\sim \sigma^2_i\chi^2_1, \quad (n-1)S^2_i\sim \sigma^2_i\chi^2_{n-1}\]
which satisfies
\[n\bar{Y}_i^2\eqd B_i \cdot (n-1)S^2_i, \]
where $B_i\sim\textnormal{Beta}(\frac{1}{2},\frac{n}{2}-1)$ is drawn independently of all the data. 
Now define
\[\tilde{p}_i = \frac{1}{m}\sum_{j=1}^m\left(1 - F_{\textnormal{Beta}(\frac{1}{2},\frac{n}{2}-1)}\left(\frac{B_i \cdot S^2_i}{S^2_j}\right)\right).\]
Under $H_{0,i}$, note that, aside from the $i$th term, this is equal in distribution to $p_i$---that is,
\[\frac{1}{m}\sum_{j\in[m]\setminus\{i\}}\left(1 - F_{\textnormal{Beta}(\frac{1}{2},\frac{n}{2}-1)}\left(\frac{B_i \cdot S^2_i}{S^2_j}\right)\right) \eqd \frac{1}{m}\sum_{j\in[m]\setminus\{i\}}\left(1 - F_{\textnormal{Beta}(\frac{1}{2},\frac{n}{2}-1)}\left(\frac{n\bar{Y}_i^2}{(n-1)S^2_j}\right)\right).\]
Therefore, for $i\in\nulls$,
\[\PP{p_i\leq t}\leq \PP{\tilde{p}_i\leq t+ \frac{1}{m}}.\]
Next, conditional on $Y_1,\dots,Y_m$, note that $B_i\cdot (n-1)S^2_i$ has right-tail CDF $\bar{F}_i(x):=1 - F_{\textnormal{Beta}(\frac{1}{2},\frac{n}{2}-1)}(\frac{x}{(n-1)S^2_i})$. By Proposition~\ref{prop:average_null_CDFs}, then,
\[\frac{1}{m}\sum_{i=1}^m \PPst{\tilde{p}_i\leq t}{Y_1,\dots,Y_m} \leq t.\]
By marginalizing over $Y_1,\dots,Y_m$, this completes the proof.
\end{proof}

\section{Extensions}\label{app:extensions}
In this section we present several additional theoretical findings that extend
some of our main results.
\subsection{The independent case, with an additional assumption}
Recall that in Theorem~\ref{thm:fdr_independent}, we showed that the BH procedure,
when run on independent compound p-values, results in an FDR bound that is at most $1.93\alpha$---and in Proposition~\ref{prop:fdr_example}, we showed that a constant-factor inflation is unavoidable, without further assumptions. In this next result, we show that, if we place
an additional assumption on the marginal distributions of the $p_i$'s (ensuring that $\PP{p_i\leq \alpha}$ cannot be too large for any $i\in\nulls$), a tighter bound is possible. This can also be derived from \citet[][Theorem~1]{10.1214/18-EJS1441}; here we present a streamlined proof for completeness.
\begin{theorem}\label{thm:fdr_with_gamma}
    In the setting of Theorem~\ref{thm:fdr_independent}, if we additionally assume that
    \[\max_{j\in\nulls}\PP{p_j\leq\alpha}\leq \gamma,\]
    then the BH procedure at level $\alpha$ satisfies
    \[\fdr\leq \frac{\alpha}{1-\gamma}.\]
\end{theorem}
\begin{proof}[Proof of Theorem~\ref{thm:fdr_with_gamma}]
Define
\[\kh_i = \max\left\{k \in [m]: 1 + \sum_{j\in[m]\setminus \{i\}} \One{p_j\leq \frac{\alpha k}{m}} \geq k\right\},\]
which is the number of rejections that the BH procedure would make, if run on values $(p_1,\dots,p_{i-1},0,p_{i+1},\dots,p_m)$. By construction, we have $\kh_i\geq 1\vee \kh$ always, and moreover it is well known that
\begin{equation}\label{eqn:kh_khi} i\in\Sbh \ \Longleftrightarrow \ \kh_i = \kh \ \Longleftrightarrow \ p_i\leq \frac{\alpha \kh_i}{m},\end{equation}
by standard leave-one-out arguments for analyzing the BH procedure \citep{benjamini2001control,ferreira2006benjamini}.
Note that $\kh_i$ is a function of $(p_j)_{j\neq i}$, and therefore is independent of $p_i$.

For each $i\in\nulls$, let $F_i(t) = \PP{p_i\leq t}$ be the CDF of $p_i$. 
We then have
\begin{align*}
    \fdr &= \EE{\fdp(\Sbh)}\\
    &=\EE{\frac{\sum_{i\in\nulls}\One{i\in\Sbh}}{1\vee|\Sbh|}}
    =\sum_{i\in\nulls}\EE{\frac{\One{p_i\leq \frac{\alpha\kh}{m}}}{1\vee \kh}}\textnormal{ by definition of $\Sbh$}\\
    &=\sum_{i\in\nulls}\EE{\frac{\One{p_i\leq \frac{\alpha\kh_i}{m}}}{\kh_i}}\textnormal{\quad by~\eqref{eqn:kh_khi}}\\
    &=\sum_{i\in\nulls}\EE{\frac{F_i\left(\frac{\alpha\kh_i}{m}\right)}{\kh_i}}\textnormal{\quad since $p_i\independent \kh_i$}\\
    &\leq\frac{1}{1-\gamma}\sum_{i\in\nulls}\EE{\frac{F_i\left(\frac{\alpha\kh_i}{m}\right)}{\kh_i} \cdot \One{p_i>\alpha}},
\end{align*}
where the last step holds since $p_i\independent \kh_i$, and $\PP{p_i>\alpha}\geq 1-\gamma$.
Next, define
\[\kh_+ = \max\left\{k \in [m]: 1 + \sum_{i=1}^m \One{p_i\leq \frac{\alpha k}{m}} \geq k\right\},\]
and observe that, if $p_i>\alpha$, then we must have $\kh_+ = \kh_i$ by construction. Therefore, 
\begin{align*}
\fdr
&\leq \frac{1}{1-\gamma}\sum_{i\in\nulls}\EE{\frac{F_i\left(\frac{\alpha\kh_i}{m}\right)}{\kh_i} \cdot \One{p_i>\alpha}}\\
& \leq \frac{1}{1-\gamma}\sum_{i\in\nulls}\EE{\frac{F_i\left(\frac{\alpha\kh_+}{m}\right)}{\kh_+} }\textnormal{ since $\kh_i=\kh_+$ whenever $\One{p_i>\alpha}=1$}\\
& = \frac{1}{1-\gamma}\EE{\frac{\sum_{i\in\nulls} F_i\left(\frac{\alpha\kh_+}{m}\right)}{\kh_+}}
\leq \frac{1}{1-\gamma}\EE{\frac{m \cdot \frac{\alpha\kh_+}{m}}{\kh_+}} = \frac{\alpha}{1-\gamma},
\end{align*}
where the last inequality holds since, for all $t\in[0,1]$,
\[\sum_{i\in\nulls}F_i(t) = \sum_{i\in\nulls}\PP{p_i\leq t}\leq mt,\]
by definition of compound p-values.
\end{proof}

\subsection{Universality result}
In Proposition~\ref{prop:average_null_CDFs}, we showed that using an average of null distributions leads to compound p-values. The following proposition shows a universality result that can be viewed as a sort of converse: the average-of-nulls construction is essentially the \emph{only} way to obtain compound p-values.
\begin{proposition}\label{prop:universality} Suppose we observe data $X_i$ for each hypothesis $H_{0,i}$, for $i\in[m]$. Let $p_1,\dots,p_m$ be compound p-values, and assume $p_i$ depends only on $X_i$.
Then there exist functions $T_i$ for each $i\in[m]$, such that
    \[p_i \geq \frac{1}{m}\sum_{j\in\nulls} F_j(T_i(X_i))\textnormal{\quad for all $i\in[m]$},\]
    where $F_j$ is the CDF of the distribution of $T_j(X_j)$ for each $j\in\nulls$.
\end{proposition}
\begin{proof}[Proof of Proposition~\ref{prop:universality}]
For each $i\in[m]$, let $T_i$ be the function such that $p_i=T_i(X_i)$. Let $F_i$ be the CDF of $p_i$ for each $i\in\nulls$. Then for any $t\in[0,1]$,
    \[\frac{1}{m}\sum_{j\in\nulls}F_j(t) = 
    \frac{1}{m}\sum_{j\in\nulls}\PP{p_j\leq t}\leq t,\]
    and so for each $i\in[m]$,
    \[\frac{1}{m}\sum_{j\in\nulls} F_j(T_i(X_i)) \leq T_i(X_i) = p_i.\]
\end{proof}

\section{Proofs of supporting lemmas}\label{app:proofs_lemmas}

\subsection{Proof of Lemma~\ref{lem:BH_condition_on_nonnulls}}
    Below, we will show that $\kh=k_I$ (where $\kh$ is defined in the construction of the BH procedure~\eqref{eqn:BH_define_kh}). Since $|\Sbh|=\kh$, this proves the claim that $|\Sbh|=k_I$.  We also have
    \[\Sbh = \left\{j : p_j\leq \frac{\alpha\kh}{m}\right\} = \left\{j : p_j\leq \frac{\alpha k_I}{m}\right\},\]
    and so
    \[\Sbh\cap\nulls =  \left\{j\in\nulls : p_j \leq \frac{\alpha k_I}{m}\right\}.\]
    But, by definition of $I$,
    \[
    I \leq \sum_{j\in\nulls}\One{p_j\leq \frac{\alpha k_I}{n}} \leq \sum_{j\in\nulls}\One{p_j\leq \frac{\alpha k_{I+1}}{n}} < I+1,
    \]
    so $|\Sbh\cap\nulls| = I$.

    Now we verify that $\kh = k_I$. First, by definition of $I$,
    \[\sum_{j\in\nulls}\One{p_j\leq \frac{\alpha k_I}{m}} \geq I,\]
    and so
    \[\sum_{j=1}^m \One{p_j\leq \frac{\alpha k_I}{m}} \geq I + 
    \sum_{j\not\in\nulls}\One{p_j\leq \frac{\alpha k_I}{m}} \geq k_I,\]
    where the last step holds by definition of $k_i$. Therefore, $\kh \geq k_I$.
    Next, let 
    \[i = \sum_{j\in\nulls}\One{p_j\leq \frac{\alpha \kh}{m}}.\]
    By definition of $\kh$, we must have 
    \[i + \sum_{j\not\in\nulls}\One{p_j\leq \frac{\alpha \kh}{m}}
    = \sum_{j=1}^m\One{p_j\leq \frac{\alpha \kh}{m}}\geq \kh,\]
    which means that $k_i \geq \kh$, by definition of $k_i$. Next,
    \[\sum_{j\in\nulls}\One{p_j\leq \frac{\alpha k_i}{m}}
    \geq \sum_{j\in\nulls}\One{p_j\leq \frac{\alpha \kh}{m}} = i,\]
    and therefore $I\geq i$, by definition of $I$. This means that $k_I\geq k_i\geq \kh$, which completes the proof.

\subsection{Proof of Lemma~\ref{lem:increasing_sets}}

We will prove a strictly stronger result:
\begin{lemma}\label{lem:increasing_sets_stronger}
    Let $Y_1,\dots,Y_k$ be real-valued random variables whose joint distribution satisfies the following condition: for any distinct $i_1,\dots,i_L\in[k]$,
    \begin{multline}\label{eqn:prds_stronger}
    \textnormal{The function }(y_1,\dots,y_L) \mapsto \PPst{(Y_1,\dots,Y_k) \in A}{Y_{i_1}=y_1,\dots,Y_{i_L}=y_L}\\
    \textnormal{ is nondecreasing coordinatewise in each $y_1,\dots,y_L$, for any increasing set $A\subseteq\R^k$.}\end{multline}
    Then for any increasing sets $A,B\subseteq\R^k$,
    \[\PP{(Y_1,\dots,Y_k)\in A\cap B}\geq \PP{(Y_1,\dots,Y_k)\in A}\cdot \PP{(Y_1,\dots,Y_k)\in B}.\]
\end{lemma}
In particular, if $Y_1,\dots,Y_k$ are mutually independent then the condition~\eqref{eqn:prds_stronger} is satisfied (meaning that Lemma~\ref{lem:increasing_sets} is a special case of this new result). For intuition, we remark that the condition~\eqref{eqn:prds_stronger} can be viewed as a stronger, multivariate version of the PRDS assumption~\eqref{eqn:PRDS}. 

\begin{proof}[Proof of Lemma~\ref{lem:increasing_sets_stronger}]
    Let $P$ denote the joint distribution of $Y=(Y_1,\dots,Y_k)$, so that our aim is to show $P(A\cap B)\geq P(A)\cdot P(B)$ for all increasing sets $A,B \subseteq \mathbb{R}^k$. We will prove the claim by induction on $k$. First, if $k=1$, then we must either have $A\subseteq B$ or $B\subseteq A$, and so the claim holds trivially.

    Now fix any $k\geq 2$. Define $A_t = \{y\in\R^{k-1} : (y,t)\in A\}$ and $B_t = \{y\in\R^{k-1} : (y,t)\in B\}$. We see that $A_t,B_t$ are each increasing sets for any fixed $t$ (inheriting this property from $A,B$). Now let
    $Y_{-k} = (Y_1,\dots,Y_{k-1})\in\R^{k-1}$.
    The conditional distribution of $Y_{-k}\mid Y_k=t$ satisfies the condition~\eqref{eqn:prds_stronger} (because $Y$ itself satisfies the condition). Therefore, by induction, we have
    \[\PPst{Y_{-k}\in A_t\cap B_t}{Y_k = t} \geq \PPst{Y_{-k}\in A_t}{Y_k = t}\cdot \PPst{Y_{-k}\in B_t}{Y_k = t}.\]
    Equivalently, we can write this as
    \[\PPst{Y\in A\cap B}{Y_k = t} \geq \PPst{Y\in A}{Y_k = t}\cdot \PPst{Y\in B}{Y_k = t},\]
    by definition of $A_t,B_t$.
    Since this holds for every $t\in\R$, we therefore have
    \[\PPst{Y\in A\cap B}{Y_k} \geq \PPst{Y\in A}{Y_k}\cdot \PPst{Y\in B}{Y_k}\]
    almost surely. Marginalizing over $Y_k$, then,
    \begin{multline*}
        \PP{Y\in A\cap B}
        =\EE{\PPst{Y\in A\cap B}{Y_k}}\\
        \geq \EE{\PPst{Y\in A}{Y_k}\cdot \PPst{Y\in B}{Y_k}}
        =\EE{f(Y_k)\cdot g(Y_k)},
    \end{multline*}
    where we let $f(t) = \PPst{Y\in A}{Y_k = t}$ and $g(t) = \PPst{Y\in B}{Y_k = t}$. By~\eqref{eqn:prds_stronger}, we know that $f$ and $g$ are both nondecreasing functions. By the rearrangement inequality \citep[e.g.][Theorem~2.2]{schmidt2014inequalities}, therefore,
    \[\EE{f(Y_k)\cdot g(Y_k)} \geq \EE{f(Y_k)}\cdot \EE{g(Y_k)}.\]
    Finally, we calculate
    \[\EE{f(Y_k)} = \EE{\PPst{Y\in A}{Y_k}} = \PP{Y\in A} ,\]
    and similarly $\EE{g(Y_k)} = \PP{Y\in B}$, which completes the proof.    
\end{proof}

\subsection{Proof of Lemma~\ref{lem:t_q_sum}}
Define a sequence $c_1\leq c_2\leq \dots$ as follows:
$c_1 = 1$, $c_2 = 1.5$, and then recursively define
\[c_\ell = c_{\ell-1} + \PP{\textnormal{Pois}\left(\frac{\ell-1}{c_{\ell-1}}\right) \geq \ell-1} - c_{\ell-1} \cdot \PP{\textnormal{Pois}\left(\frac{\ell-1}{c_{\ell-1}}\right) \geq \ell}\]
for each $\ell\geq3$.
Now fix any feasible $t_1,\dots,t_L$ and $\bar{q}_1,\dots,\bar{q}_L$. We will prove the following bound by induction: for each $\ell=1,\dots,L$,
\begin{equation}\label{eqn:c_ell_step}\sum_{i=1}^{\ell}  \frac{i}{t_i}\cdot \bar{q}_i \prod_{j=i+1}^\ell (1-\bar{q}_j)\leq c_\ell.\end{equation}
Assuming that this is true, then at $\ell=L$ we have
\[\sum_{i=1}^L  \frac{i}{t_i}\cdot \bar{q}_i\prod_{j=i+1}^L(1-\bar{q}_j) \leq c_L.\]
In Section~\ref{sec:monotonicity_c_ell} below, we will prove that this sequence is increasing and bounded; we can verify numerically that $\lim_{\ell\to\infty} c_\ell \leq 1.9227$, and this will complete the proof.

Now we verify~\eqref{eqn:c_ell_step}.
First, at $\ell=1$,
\[\frac{1}{t_1}\cdot \bar{q}_1 \leq \frac{1}{t_1}\cdot B_1(t_1).\]
And, by a union bound,
\begin{align*}
    B_1(t_1)
    &=\sup\left\{\PP{A_1+\dots+A_r\geq 1}: \EE{A_1 + \dots + A_r}\leq t_1\right\}\\
    &\leq\sup\left\{\PP{A_1=1}+\dots+\PP{A_r=1}: \EE{A_1 + \dots + A_r}\leq t_1\right\}\\
    &=t_1.
\end{align*}
Therefore $\frac{1}{t_1}\cdot B_1(t_1)\leq 1=c_1$, which proves~\eqref{eqn:c_ell_step} for the case $\ell=1$.

Next let $\ell=2$. Since we have proved the bound at $\ell=1$, we have
\begin{align*}
    \frac{1}{t_1}\cdot \bar{q}_1(1-\bar{q}_2) + \frac{2}{t_2}\cdot \bar{q}_2
    &\leq 1\cdot (1-\bar{q}_2) + \frac{2}{t_2}\cdot \bar{q}_2\\
    &=1 + \left(\frac{2}{t_2} - 1\right)\bar{q}_2\\
    &\leq 1 + \left(\frac{2}{t_2} - 1\right)_+ B_2(t_2).
\end{align*}
And, by a union bound,
\begin{align*}
    B_2(t_2)
    &=\sup\left\{\PP{A_1+\dots+A_r\geq 2}: \EE{A_1 + \dots + A_r}\leq t_2\right\}\\
    &\leq \sup\left\{\sum_{1\leq j<j'\leq r}\PP{A_j=A_{j'}=1}: \EE{A_1 + \dots + A_r}\leq t_2\right\}\\
    &=\sup\left\{\sum_{1\leq j<j'\leq r}\PP{A_j=1}\PP{A_{j'}=1}: \EE{A_1 + \dots + A_r}\leq t_2\right\}\\
    &\leq \sup\left\{\frac{1}{2}\left(\sum_{j=1}^r \PP{A_j=1}\right)^2: \EE{A_1 + \dots + A_r}\leq t_2\right\}\\
    &=\frac{1}{2}t_2^2.
\end{align*}
And,
\[\left(\frac{2}{t_2} - 1\right)_+\cdot \frac{1}{2}t_2^2 \leq \frac{1}{2},\]
by simply maximizing the quadratic over $t_2$. 
Therefore,
\[\frac{1}{t_1}\cdot \bar{q}_1(1-\bar{q}_2) + \frac{2}{t_2}\cdot \bar{q}_2\leq 1.5=c_2,\]
which proves~\eqref{eqn:c_ell_step} for $\ell=2$.

Next fix any $\ell\geq 3$. First, 
\begin{align*}\sum_{i=1}^{\ell}  \frac{i}{t_i}\cdot \bar{q}_i \prod_{j=i+1}^\ell (1-\bar{q}_j)
&=\sum_{i=1}^{\ell-1}  \frac{i}{t_i}\cdot \bar{q}_i \prod_{j=i+1}^{\ell-1} (1-\bar{q}_j) \cdot (1-\bar{q}_\ell) + \frac{\ell}{t_\ell}\cdot \bar{q}_\ell\\
&\leq c_{\ell-1} \cdot (1-\bar{q}_\ell) + \frac{\ell}{t_\ell}\cdot \bar{q}_\ell\\
&=c_{\ell-1} + \left(\frac{\ell}{t_\ell} - c_{\ell-1}\right) \bar{q}_\ell\\
&\leq c_{\ell-1} + \left(\frac{\ell}{t_\ell} - c_{\ell-1}\right)_+ B_\ell(t_\ell),
\end{align*}
where the first inequality holds by induction.
Next, if $\left(\frac{\ell}{t_\ell} - c_{\ell-1}\right)_+>0$, then we must have
\[t_\ell \leq \frac{\ell}{c_{\ell-1}} \leq \frac{\ell}{1.5}\leq \ell-1,\]
where the last steps hold since $c_{\ell-1}\geq c_2=1.5$ and $\ell\geq 3$. In this case, it holds that
\[
    B_\ell(t_\ell)
    =\sup\left\{\PP{A_1+\dots+A_r\geq \ell}: \EE{A_1 + \dots + A_r}\leq t_\ell\right\}\\
    \leq \PP{\textnormal{Pois}(t_\ell)\geq \ell},
\]
by \citet[Theorem 4]{hoeffding1956distribution}. 
So far, then, we have proved that
\[\sum_{i=1}^{\ell}  \frac{i}{t_i}\cdot \bar{q}_i \prod_{j=i+1}^\ell (1-\bar{q}_j)
\leq c_{\ell-1} + \left(\frac{\ell}{t_\ell} - c_{\ell-1}\right)_+ \cdot \PP{\textnormal{Pois}(t_\ell)\geq \ell}.\]
Next, we calculate
\[
\frac{\ell}{t_\ell} \cdot \PP{\textnormal{Pois}(t_\ell)\geq \ell}
= \sum_{j=\ell}^\infty \frac{\ell}{t_\ell} \cdot \frac{t_\ell^j e^{-t_\ell}}{j!} \leq  \sum_{j=\ell}^\infty \frac{t_\ell^{j-1}e^{-t_\ell}}{(j-1)!} = \PP{\textnormal{Pois}(t_\ell)\geq \ell-1},\]
and therefore,
\[\sum_{i=1}^{\ell}  \frac{i}{t_i}\cdot \bar{q}_i \prod_{j=i+1}^\ell (1-\bar{q}_j)
\leq c_{\ell-1} + \Big(\PP{\textnormal{Pois}(t_\ell)\geq \ell-1}  - c_{\ell-1} \cdot \PP{\textnormal{Pois}(t_\ell)\geq \ell}\Big)_+\leq c_\ell,\]
where for the last step, we use the fact that
\[\sup_{t>0}\Big(\PP{\textnormal{Pois}(t)\geq \ell-1}  - c_{\ell-1} \cdot \PP{\textnormal{Pois}(t)\geq \ell}\Big)\]
is attained at $t=\frac{\ell-1}{c_{\ell-1}}$ (and is nonnegative),
by Lemma~\ref{lem:Poisson_difference_sup}. This verifies~\eqref{eqn:c_ell_step} for the $\ell$th term, and therefore completes the proof.

\begin{lemma}\label{lem:Poisson_difference_sup}
    Fix any $c>0$ and any integer $i\geq 2$, and define the function $f:(0,\infty)\to \R$ as
    \[f(t) = \PP{\textnormal{Pois}(t)\geq i-1} - c\cdot \PP{\textnormal{Pois}(t)\geq i} .\]
    Then $f(\cdot)$ is maximized at $t=\frac{i-1}{c}$.
\end{lemma}

\begin{proof}[Proof of Lemma~\ref{lem:Poisson_difference_sup}]
    We compute the derivative:
    \begin{align*}
    f'(t) &= \frac{\partial}{\partial t}\Big\{\PP{\textnormal{Pois}(t)\geq i-1} - c\cdot  \PP{\textnormal{Pois}(t)\geq i}\Big\}\\
    &= \frac{\partial}{\partial t}\left\{\sum_{j=i-1}^\infty \frac{t^je^{-t}}{j!} - c\cdot \sum_{j=i}^\infty \frac{t^je^{-t}}{j!}\right\}\\
    &= \sum_{j=i-1}^\infty \frac{t^{j-1}e^{-t}}{(j-1)!} - \sum_{j=i-1}^\infty\frac{t^je^{-t}}{j!} - c\cdot \left(\sum_{j=i}^\infty \frac{t^{j-1}e^{-t}}{(j-1)!} - \frac{t^je^{-t}}{j!}\right)\\
    &=\frac{t^{i-2}e^{-t}}{(i-2)!} - c\cdot \frac{t^{i-1}e^{-t}}{(i-1)!} \\
    &= \left(\frac{i-1}{t}-c\right) \cdot \frac{t^{i-1}e^{-t}}{(i-1)!}\\
    &= \left(\frac{i-1}{t}-c\right) \cdot \PP{\textnormal{Pois}(t) = i-1}.
\end{align*}
We can see that $f'(t)=0$ at $t=\frac{i-1}{c}$. Moreover, since $\PP{\textnormal{Pois}(t) = i-1}>0$, this verifies that $f'(t)$ is positive for $t<\frac{i-1}{c}$ and negative for $t>\frac{i-1}{c}$, which completes the proof.
\end{proof}

\subsection{Monotonicity of the sequence \texorpdfstring{$(c_\ell)$}{(c_\ell)}}\label{sec:monotonicity_c_ell}

\begin{lemma}
The sequence $(c_\ell)$ defined in the proof of Lemma~\ref{lem:t_q_sum} is strictly increasing with $\lim_{\ell \rightarrow \infty} c_\ell < \infty$.
\end{lemma}
\begin{proof}
From the representation
\[
c_\ell = c_{\ell-1}\PP{\textnormal{Pois}\left(\frac{\ell-1}{c_{\ell-1}}\right) \leq \ell - 1} + \PP{\textnormal{Pois}\left(\frac{\ell-1}{c_{\ell-1}}\right) \geq \ell - 1},
\]
we see immediately by induction that $c_\ell > 1$ for $\ell \geq 2$.  Now fix $\ell \geq 3$, let $c = c_{\ell-1}$, let $\lambda = (\ell-1)/c$ and let $Z \sim \textnormal{Pois}(\lambda)$.  We have
\begin{align*}
    \PP{Z \geq \ell} &= \PP{Z = \ell-1}\sum_{r=1}^\infty \frac{\PP{Z = \ell-1+r}}{\PP{Z = \ell-1}} = \PP{Z = \ell-1}\sum_{r=1}^\infty \prod_{s=1}^r \frac{\lambda}{\ell - 1 + s} \\
    &< \PP{Z = \ell-1}\sum_{r=1}^\infty \Bigl(\frac{\lambda}{\ell}\Bigr)^r = \PP{Z = \ell-1} \cdot \frac{\lambda}{\ell - \lambda},
\end{align*}
where we have used the fact that $\lambda/\ell < 1/c < 1$ in the final step.  Thus
\begin{align*}
  c_\ell - c_{\ell-1} = \PP{Z \geq \ell - 1} - c \, \PP{Z \geq \ell} &= \PP{Z = \ell-1} - (c-1)\PP{Z \geq \ell} \\
  &> \PP{Z = \ell-1} \cdot \frac{\ell - c\lambda}{\ell - \lambda} > 0,
\end{align*}
which establishes the first claim of the lemma.

For the second claim, for $\ell \geq 2$, let $\lambda_\ell = \ell/c_\ell$ and $Z_\ell \sim \textnormal{Pois}(\lambda_\ell)$, so that 
\[
c_{\ell+1} - c_\ell = \PP{Z_\ell = \ell} - (c_\ell  - 1)\PP{Z_\ell \geq \ell+1} < \PP{Z_\ell = \ell}.
\]
Now, by the first part of the lemma, $\lambda_\ell \leq 2\ell/3$ for $\ell \geq 2$, so since $\lambda \mapsto \PP{\textnormal{Pois}(\lambda)=\ell}$ is strictly increasing on $(0,\ell)$, we see that
\[
c_{\ell+1} - c_\ell < e^{-2\ell/3}\frac{(2\ell/3)^\ell}{\ell!}.
\]
It follows that for every $L \geq 3$,
\begin{align*}
c_L = c_2 + \sum_{\ell=2}^{L-1} (c_{\ell+1} - c_\ell) < \frac{3}{2} + \sum_{\ell=2}^{L-1} e^{-2\ell/3}\frac{(2\ell/3)^\ell}{\ell!} \leq \frac{3}{2} + \frac{1}{\sqrt{2\pi}} \sum_{\ell=2}^{L-1} \frac{(2e^{1/3}/3)^\ell}{\ell^{1/2}}, 
\end{align*}
where we used the Stirling's formula bound $\ell! \geq (2\pi\ell)^{1/2}(\ell/e)^\ell$ in the final step.  Since $2e^{1/3}/3 < 1$, we conclude that $\lim_{\ell \rightarrow \infty} c_\ell < \infty$.
\end{proof}

\subsection{Proof of Lemma~\ref{lem:prob_bound_for_globalnull}}
By construction, 
\[(s_1,\dots,s_m)\mapsto \PP{\mathcal{U}_1(s_1)\cup\dots\cup\mathcal{U}_m(s_m)}\]
is coordinatewise nondecreasing in each $s_{i,j}$. Therefore, the supremum 
\[\sup\left\{ \PP{\mathcal{U}_1(s_1)\cup\dots\cup\mathcal{U}_m(s_m)}  \ : \ s_1,\dots,s_m \in[0,1]^m,  \ \sum_{j=1}^m s_{i,j}\leq \alpha i  \ \forall  \ i\right\}\]
can equivalently be written as
\[\sup\left\{ \PP{\mathcal{U}_1(s_1)\cup\dots\cup\mathcal{U}_m(s_m)}  \ : \ s_1,\dots,s_m \in[0,1]^m,  \ \sum_{j=1}^m s_{i,j} =  \alpha i  \ \forall  \ i\right\},\]
where the inequality constraint $\sum_{j=1}^m s_{i,j}\leq \alpha i$ has now been replaced by equality, $\sum_{j=1}^m s_{i,j}= \alpha i$. 

Now choose any $s_1,\dots,s_m$ satisfying these constraints. We have
\begin{multline*}\PP{\mathcal{U}_1(s_1)\cup\dots\cup\mathcal{U}_m(s_m)}
= \PP{\mathcal{U}_1(s_1)} + \PP{\mathcal{U}_1(s_1)^c\cap\big(\mathcal{U}_2(s_2)\cup\dots\cup\mathcal{U}_m(s_m)\big)} \\\leq  \PP{\mathcal{U}_1(s_1)} + \sum_{i=2}^m \PP{\mathcal{U}_1(s_1)^c\cap \mathcal{U}_i(s_i)},\end{multline*}
by a union bound. We also calculate
\[\PP{\mathcal{U}_1(s_1)} = \PP{U_1\leq s_{1,1} \textnormal{ or }\dots \textnormal{ or }U_m\leq s_{1,m}} \leq \sum_{j=1}^m \PP{U_j\leq s_{1,j}} = \sum_{j=1}^m s_{1,j} = \alpha,\]
again by a union bound.
And, for each $i\in\{2,\dots,m\}$,
\begin{multline*}\PP{\mathcal{U}_1(s_1)^c\cap \mathcal{U}_i(s_i)} = \PP{U_j>s_{1,j} \textnormal{ for all $j\in[m]$, and }\sum_{j=1}^m \One{U_j\leq s_{i,j}}\geq i}\\
\leq \PP{\sum_{j=1}^m \One{s_{1,j} < U_j\leq s_{i,j}}\geq i} .\end{multline*}
Note that $\sum_{j=1}^m \One{s_{1,j} < U_j\leq s_{i,j}}$ is a sum of independent Bernoulli random variables, where the $j$th random variable has probability $s_{i,j}-s_{1,j}$ of success, and moreover $\sum_{j=1}^m (s_{i,j}-s_{1,j}) = \alpha (i-1)$ by our constraints on $s_1,\dots,s_m$.
We therefore have
\[\PP{\sum_{j=1}^m \One{s_{1,j} < U_j\leq s_{i,j}}\geq i} \leq \PP{\textnormal{Pois}(\alpha(i-1)) \geq i},\]
by \citet[Theorem 4]{hoeffding1956distribution}, where we use the fact that $\alpha(i-1) + 1\leq i$. Combining all our calculations so far, we have
\begin{multline*}\PP{\mathcal{U}_1(s_1)\cup\dots\cup\mathcal{U}_m(s_m)}\leq \alpha + \sum_{i=2}^m \PP{\textnormal{Pois}(\alpha(i-1)) \geq i}\\ \leq \alpha + \sum_{i=1}^\infty \PP{\textnormal{Pois}(\alpha i) > i} = \alpha + \frac{\frac{1}{2}\alpha^2}{(1-\alpha)^2},\end{multline*}
where the last step holds by the following lemma:
\begin{lemma}\label{lem:Poisson_identities}
    For any $t\in(0,1)$,
    \[\sum_{k=1}^\infty \PP{\textnormal{Pois}(tk) \geq k} = \frac{t - \frac{1}{2}t^2}{(1-t)^2}\]
    and
    \[\sum_{k=1}^\infty \PP{\textnormal{Pois}(tk) > k} =  \frac{\frac{1}{2}t^2}{(1-t)^2}.\]
\end{lemma}

Finally, if $\alpha\leq 0.5$, then we have $\alpha + \frac{\frac{1}{2}\alpha^2}{(1-\alpha)^2}\leq \alpha+ 2\alpha^2$, while if $\alpha>0.5$ then the result holds trivially since $\alpha+ 2\alpha^2>1$.

\begin{proof}[Proof of Lemma~\ref{lem:Poisson_identities}]
We will use the following standard facts: first, the Poisson distribution can be related to the Gamma distribution via
\begin{equation}
\label{eqn:poisson_vs_gamma}
\PP{\textnormal{Pois}(\lambda)\geq\ell} = \PP{\textnormal{Gamma}(\ell,1)\leq\lambda}
\end{equation}
for all $\lambda > 0$, $\ell \in \mathbb{N}$ \citep[e.g.][p.~190]{johnson1992univariate}; and second, the incomplete Gamma function $\gamma(s,x) = \Gamma(s)\cdot \PP{\textnormal{Gamma}(s,1)\leq x}$ can be expressed as the series
\begin{equation}\label{eqn:gamma_series}
\gamma(s,x) = x^s\sum_{i=0}^\infty \frac{(-x)^i}{i!(s+i)}
\end{equation}
for $s,x > 0$ \citep[][Eq.~(8.7.2)]{paris2010incomplete}.  Combining the two facts~\eqref{eqn:poisson_vs_gamma} and~\eqref{eqn:gamma_series}, we obtain for $k \in \mathbb{N}$ that
\[
\PP{\textnormal{Pois}(tk)\geq k} = \frac{\gamma(k,tk)}{\Gamma(k)} = (tk)^k \sum_{i=0}^\infty \frac{(-tk)^i}{i!(k+i)}\cdot \frac{1}{\Gamma(k)} =k(tk)^k \sum_{i=0}^\infty \frac{(-tk)^i}{i!(k+i)}\cdot \frac{1}{k!},
\]
since $\Gamma(k)=(k-1)!$. Therefore,
\begin{align*}
    \sum_{k=1}^\infty \PP{\textnormal{Pois}(tk)\geq k}
    &=\sum_{k=1}^\infty k(tk)^k \sum_{i=0}^\infty \frac{(-tk)^i}{i!(k+i)}\cdot \frac{1}{k!}\\
    &=\sum_{r=1}^\infty \sum_{k=1}^r k(tk)^k \frac{(-tk)^{r-k}}{(r-k)!r}\cdot \frac{1}{k!}\textnormal{\quad by setting $r=k+i$}\\
    &=\sum_{r=1}^\infty \frac{t^r}{r}\sum_{k=1}^r \frac{k^{r+1}(-1)^{r-k}}{k!(r-k)!}\\
    &=\sum_{r=1}^\infty \frac{t^r}{r} \cdot \binom{r+1}{2} 
    =\sum_{r=1}^\infty t^r \cdot \frac{r+1}{2} = \frac{t - \frac{1}{2}t^2}{(1-t)^2},
\end{align*}
where to move to the last line, we use the fact that $\sum_{k=1}^r \frac{(-1)^{r-k}k^{r+1}}{k!(r-k)!} = \stirling{r+1}{r} = \binom{r+1}{2}$, where $\stirling{a}{b}$ denotes a Stirling number of the second kind \citep[][p.7]{boyadzhiev2009exponential}.  

Next, again combining the two facts~\eqref{eqn:poisson_vs_gamma} and~\eqref{eqn:gamma_series}, we have
\[\PP{\textnormal{Pois}(tk)> k} = \PP{\textnormal{Pois}(tk)\geq k+1} = \frac{\gamma(k+1,tk)}{\Gamma(k+1)}  =(tk)^{k+1} \sum_{i=0}^\infty \frac{(-tk)^i}{i!(k+i+1)}\cdot \frac{1}{k!}.\]
Therefore,
\begin{align*}
    \sum_{k=1}^\infty \PP{\textnormal{Pois}(tk)> k}
    &=\sum_{k=1}^\infty (tk)^{k+1} \sum_{i=0}^\infty \frac{(-tk)^i}{i!(k+i+1)}\cdot \frac{1}{k!}\\
    &=\sum_{r=1}^\infty \sum_{k=1}^r (tk)^{k+1} \frac{(-tk)^{r-k}}{(r-k)!(r+1)}\cdot \frac{1}{k!} \textnormal{\quad by setting $r=k+i$}\\
    &=\sum_{r=1}^\infty \frac{t^{r+1}}{r+1}\sum_{k=1}^r \frac{k^{r+1}(-1)^{r-k}}{k!(r-k)!}\\
    &=\sum_{r=1}^\infty \frac{t^{r+1}}{r+1} \cdot \binom{r+1}{2} 
    =\sum_{r=1}^\infty t^{r+1} \cdot \frac{r}{2} = \frac{\frac{1}{2}t^2}{(1-t)^2},
\end{align*}
where to move to the last line, we again use the fact that $\sum_{k=0}^r \frac{(-1)^{r-k}k^{r+1}}{k!(r-k)!}= \stirling{r+1}{r} = \binom{r+1}{2}$.
\end{proof}

\end{document}